\documentclass[12pt]{article}
\usepackage[a4paper,margin=1in]{geometry}
\usepackage[T1]{fontenc}
\usepackage{lmodern}
\usepackage{setspace}
\usepackage{placeins}
\usepackage{amsmath,amssymb,mathtools}
\usepackage{amsthm}

\newtheorem{theorem}{Theorem}

\newtheorem{lemma}{Lemma}

\newcommand{\E}{\mathbb{E}}

\usepackage{algorithm}
\usepackage{algorithmic}

\usepackage{graphicx}
\usepackage[font=small,labelfont=bf]{caption}
\usepackage{subcaption}
\usepackage{booktabs}
\usepackage{array}
\usepackage{makecell}

\usepackage[numbers,sort&compress]{natbib}

\usepackage{comment}
\usepackage{appendix}
\usepackage{xurl}
\usepackage[
    colorlinks=true,
    linkcolor=blue,
    citecolor=blue,
    urlcolor=blue,
    breaklinks=true
]{hyperref}
\begin{document}
\title{Zeroth-Order Langevin Monte Carlo via SPSA under Noisy Function Measurements}

\author{
Hongbo Li%
\thanks{Hongbo Li is with the Department of Statistics and
Probability, Michigan State University, East Lansing, MI 48824,
USA. Email: lihongb1@msu.edu. Permanent email:
hongbo.24li@gmail.com.}
\quad and \quad
James C. Spall%
\thanks{James C. Spall is with the Johns Hopkins University
Applied Physics Laboratory and the Department of Applied
Mathematics and Statistics, Johns Hopkins University, Baltimore,
MD 21218, USA. Email: James.Spall@jhuapl.edu.}
}
\date{}  
\maketitle

\begin{abstract}
In sampling problems, gradient-based schemes such as Langevin Monte Carlo (LMC) mix faster than non-gradient-based methods, but their applicability is limited by access to the gradient of the target log-density. In practice, gradients are often unavailable and function evaluations are noisy---e.g., stochastic simulators or black-box simulators, so we propose LMC-SPSA with noise, which approximates the gradient of the target log-density using two noisy function evaluations per iteration. We prove, under noisy gradient estimates, that LMC-SPSA converges in distribution by proving the convergence in Wasserstein distance. We derive a finite-sample bound on the 2-Wasserstein distance ($W_2$) for a fixed stepsize and sharpen the dimension dependence of the bound from \(O(p^4)\) to \(O(p^2)\), where \(p\) is the dimension, and support this analysis with numerical results. We show that LMC-SPSA achieves $W_2$-accuracy $\epsilon$ with total noisy-oracle complexity of $\widetilde{O}(p/\epsilon^2+\delta^2p^3/\epsilon^3)$, where \(2\delta^2\) bounds the conditional mean square of the noise difference and \(\widetilde{O}\) suppresses logarithmic factors. This improves the noise-dependent accuracy scaling relative to the ZO-LMC method of Roy et al. We then construct a diminishing schedule for both the stepsize and the perturbation magnitude and prove that the $W_2$ converges to zero at rate \(O(k^{-1/3})\) after \(k\) iterations. Finally, we compare LMC-SPSA with ZO-LMC under the same function-evaluation budget with noise, and observe smaller empirical sampling errors for LMC-SPSA in our experiments.
\end{abstract}

\section{Introduction}
Markov chain Monte Carlo (MCMC) methods are a cornerstone of modern statistical computation and stochastic simulation, enabling sampling from complex, high-dimensional distributions when normalization constants are unknown. Since the seminal Metropolis--Hastings developments, MCMC has become pervasive across statistics and machine learning, with deep monographs and handbooks documenting both theory and practice \cite{brooks2011handbook,robert2004montecarlo,hastings1970,metropolis1953}. MCMC underpins popular application areas such as Bayesian inference \cite{jiang2023sbi}, particle methods for state space models \cite{andrieu2010pmcmc}, probabilistic programming (e.g., Stan) \cite{carpenter2017stan}, and domain sciences including phylogenetics (MrBayes) \cite{ronquist2003mrbayes}.

Classical Markov chain Monte Carlo begins with the \emph{Metropolis--Hastings (MH)} algorithm, which serves as a foundational baseline. 
Building on MH, \emph{gradient-based} samplers---such as \emph{Hamiltonian Monte Carlo (HMC)}~\cite{neal2011hmc}, \emph{Stochastic Gradient Langevin Dynamics (SGLD)}~\cite{welling2011sgld}, and the Metropolis-adjusted/Unadjusted Langevin family (MALA/ULA) \cite{roberts1996langevin,roberts1998optimal,cheng2018underdamped}---often achieve faster mixing in high dimensions when the gradient vector $\nabla \log \pi(x)$ is available, where $\pi(x)$ denotes the target density. In particular, Langevin Monte Carlo (LMC) typically burns in faster than non-gradient methods such as Metropolis--Hastings by exploiting the local geometry of $\log \pi(x)$ \cite{sun23}.

However, computing $\nabla \log \pi(x)$ is often infeasible in simulation-based inference, reinforcement learning, or empirical modeling where the target distribution is a \emph{black box}: one can query (noisy) log-likelihoods but not their gradients. This motivates \emph{zeroth-order} Langevin methods that operate only with noisy function evaluations. In such settings, one replaces $\nabla \log \pi(x)$ with a \emph{stochastic-approximation-based} gradient surrogate constructed from noisy function values, i.e., a random estimator whose expectation approximates the true gradient and whose variance is controlled by a tunable perturbation scale.
A classical construction is \emph{coordinatewise finite differences (FDSA)}~\cite{chartrand2011}, which requires $2p$ function evaluations per iteration in $p$ dimensions.
We focus on \emph{simultaneous perturbation SA (SPSA)}, which perturbs all coordinates with independent signs and uses only \emph{two} evaluations per iteration to form a stochastic gradient estimate \cite{spall92,spall94,spall05}. Under a fixed function-evaluation budget, SPSA can attain a smaller mean squared error (MSE) for the resulting gradient estimator than coordinatewise finite differences in noisy regimes \cite{blakney2019fdvsspsa}; formal comparisons with random-direction SA (RDSA) further identify regimes where SPSA achieves smaller asymptotic MSE, supported by numerics \cite{chen2021arxiv,peng2023acc}.

Beyond FDSA/RDSA/SPSA, there are \emph{first-order} Langevin variants that reduce per-iteration cost by using \emph{coordinate-wise gradients} $(\partial f/\partial x_i)$. 
For example, Ding and Li develop \emph{random-coordinate} LMC with variance-reduction, updating along randomly chosen coordinates while still requiring access to coordinatewise gradients \cite{dingli2021rcdvr}. There also exist zeroth-order samplers based on \emph{one-point} gradient surrogates: Spall proposed a one-measurement variant of SPSA that forms a (asymptotically) unbiased gradient surrogate using only one noisy function evaluation per iteration—halving the oracle calls relative to standard two-measurement SPSA \cite{spall1997onemeasurement}. In a similar spirit, Liu and Wang \cite{liu2020onepointzosgld} propose a one-point zeroth-order stochastic gradient Langevin dynamics (ZSGLD) that uses a single noisy function evaluation to form a (biased) gradient estimator within SGLD.

In many simulation- or privacy-constrained applications \cite{jiang2023sbi, pavliotis2022, bao2025}, the oracle returns only \emph{noisy function evaluations} (e.g., noisy log-likelihoods), while exact or unbiased gradients are unavailable. This motivates zeroth-order gradient surrogates tailored to noisy measurements. A parallel but distinct line of work learns a score function and then uses the learned proxy in Langevin dynamics. In score-based generative modeling, Song and Ermon~\cite{song2019score} train a score network via score matching and use the learned scores in annealed Langevin dynamics. Unlike the function-value zeroth-order setting considered here, this approach does not construct gradient estimates from noisy pointwise evaluations of the target potential. Their score network is trained on data samples perturbed by Gaussian
noise.

Motivated by the need for cost-efficient, noise-robust zeroth-order surrogates, we propose \emph{LMC-SPSA}: each Langevin step uses a two-point SPSA surrogate to form a (nearly) unbiased gradient approximation whose oracle cost is independent of dimension $p$. This contrasts with 
first-order random-coordinate LMC \cite{dingli2021rcdvr}, which still requires access to coordinatewise derivatives. Our analysis follows the line of LMC with inexact/noisy gradients admitting non-asymptotic guarantees under controlled gradient error and stepsizes \cite{dalalyan19}. Within the Hamiltonian Monte Carlo (HMC) family, SPSA has also been used to approximate gradients: Hiruntiaranakul \cite{hiruntiaranakul2024spsahmc} analyzes SPSA-HMC and two variance–reduction variants, which are complementary to our Langevin focus. Throughout this paper, oracle cost refers to the total number of noisy function evaluations. We use oracle complexity to describe
how the cost required to achieve a target accuracy scales with the parameters.

The main contribution of the work is summarized as follows:
\begin{itemize}
  \item We derive an explicit non-asymptotic bound on the Wasserstein ($W_2$)
  error after $k$ iterations of LMC-SPSA with noisy function evaluations, a fixed stepsize, and a fixed perturbation size.

  \item We sharpen the dominant dimension dependence in the $W_2$-error bound from $O(p^4)$ in \cite{sun23} to $O(p^2)$ in $p$ dimensions.

  \item We use this bound to derive the \emph{oracle complexity} of
  LMC-SPSA: $\widetilde O(p/\epsilon^2+\delta^2p^3/\epsilon^3)$ noisy function evaluations suffice to achieve a $W_2$-error of at most $\epsilon$.
  Here \(2\delta^2\) bounds the conditional mean square of the measurement noise difference in each SPSA estimate, and \(\widetilde O\) suppresses logarithmic factors.
  This improves the noise-dependent accuracy scaling of ZO-LMC in \cite{Roy2022} from \(\epsilon^{-4}\) to \(\epsilon^{-3}\). If both evaluations share the same additive noise ($\delta=0$), the oracle complexity reduces to $\widetilde{O}(p/\epsilon^2)$. 

  \item We design a diminishing stepsize schedule that preserves $W_2$-convergence, removing the constant-stepsize restriction in \cite{sun23}.

  \item We establish vanishing \(W_2\) error under diminishing
  stepsize and perturbation sequences, and derive a balanced schedule attaining
  \(W_2(\nu_k,\pi)=
  \mathcal O(\delta^{2/3}p\,k^{-1/3})\) for fixed
  \(\delta>0\), where \(k\) denotes the
  iteration index.
  
  \item We also provide oracle-call-matched numerical comparisons with the ZO-LMC and LMC-FDSA,
  showing that LMC-SPSA attains smaller empirical sampling errors under the same function-evaluation budget.
  
\end{itemize}

\section{Preliminaries \& Baselines}
\label{sec:prelim-baselines}
\paragraph{Noisy zeroth-order oracle.}
In this work, given a density function $\pi:\mathbb{R}^p \to \mathbb{R}$ with potential
function $f:\mathbb{R}^p \to \mathbb{R}$ satisfying $\pi(x)\propto e^{-f(x)}$, we study
the problem of sampling when the potential $f$ is not directly available and can only be
accessed through noisy function evaluations. We focus on an additive-noise oracle
model
\[
\widehat f(x)=f(x)+\zeta,
\]
where $\zeta$ denotes the noise terms. We refer to
this setting as \emph{stochastic zeroth-order sampling}. This is the regime relevant to our
LMC-SPSA method; see
also Section 4 in \cite{Roy2022} for a related discussion of stochastic zeroth-order sampling models.
Such noisy function-value oracles arise, for example, in simulation-based inference \cite{ pavliotis2022} where log-likelihood evaluations are obtained via stochastic simulators.

The starting point is the overdamped Langevin diffusion
\[
dL_t = -\nabla f(L_t)\,dt+\sqrt{2}\,dB_t,
\]
which has \(\pi\) as its invariant distribution under standard regularity
conditions. Applying the Euler--Maruyama discretization with stepsize \(h_k\)
gives the unadjusted Langevin update
\[
X_{k+1}
=
X_k-h_k\nabla f(X_k)+\sqrt{2h_k}\,\xi_k,
\qquad
\xi_k \overset{\mathrm{i.i.d.}}{\sim} N(0,I_p).
\]
In the zeroth-order setting, \(\nabla f(X_k)\) is unavailable and is replaced by
a gradient surrogate \(\widehat G_{k}(X_k)\) constructed from noisy function
evaluations. The zeroth-order Langevin method constructs a Markov chain $\{X_k\}$ via the recursion:
\begin{equation}\label{eq:lmc-spsa-update}
  X_{k+1} \;=\; X_k \;-\; h_k\,\widehat G_k(X_k) \;+\; \sqrt{2h_k}\,\xi_k,
  \qquad \xi_k \overset{\mathrm{i.i.d.}}{\sim} N(0,I_p),\quad k=0,1,2,\ldots
\end{equation}
where $h_k>0$ is the stepsize at iteration $k$, and $\xi_k$ is an independent $p$-dimensional standard Gaussian noise.

\paragraph{Coordinate formulas for FDSA/SPSA.}
Let $c_k>0$, and let $e_i$ be the $i$th coordinate basis vector.
For a fixed coordinate $i$, finite-difference stochastic approximation (FDSA) and simultaneous perturbation stochastic approximation (SPSA) define the following coordinatewise gradient estimators:
\begin{align}
\text{FDSA:}~~
&\widehat G^{\text{FD}}_{ki}(X_k)
=\frac{\widehat f(X_k+c_k e_i)-\widehat f(X_k-c_k e_i)}{2c_k},
\label{eq:fdsa}
\\[-0.25em]
\text{SPSA:}~~
&\widehat G^{\text{SP}}_{ki}(X_k)
=\frac{\widehat f(X_k+c_k \Delta_k)-\widehat f(X_k-c_k \Delta_k)}{2c_k\,\Delta_{ki}},
\label{eq:spsa}
\end{align}
where $\Delta_k\in\{\pm1\}^p$ is a random perturbation vector. Throughout the main analysis we take
\(\Delta_{ki}\), \(i=1,\ldots,p\), to be independent Rademacher random
variables, so that
\[
\mathbb P(\Delta_{ki}=1)=\mathbb P(\Delta_{ki}=-1)=\frac12 .
\] In SPSA, the coordinates $\{\Delta_{ki}\}$ are i.i.d.\ Bernoulli $\pm1$ in most implementations; variants using
\emph{U-shaped} or \emph{split-uniform} distributions (placing negligible mass near $0$) have also been explored
in control/NN applications \cite{spall05,cao2011acc,maeda1997}. These choices satisfy the finite inverse-moment condition $\,\E|\Delta_{ki}|^{-1}<\infty\,$ required by \eqref{eq:spsa}. To form the full gradient estimator at iteration $k$, FDSA requires $2p$ noisy function evaluations, since the two-point difference must be computed separately for each coordinate $i=1,\dots,p$. In contrast, SPSA uses the same two noisy function evaluations across all coordinates and therefore requires only two noisy function evaluations per iteration.

\paragraph{Roy et al. \cite{Roy2022}\ ZO-LMC baseline (Section~4 / Theorem~4.2).}
We compare LMC-SPSA with the ZO-LMC method of
Roy et al.~\cite[Section~4, Theorem~4.2]{Roy2022}. Their estimator uses a two-point, one-sided difference with
independent additive noise in the two function evaluations. In this paper, ``two-point'' refers to the number of query
points used in each difference. Our analysis allows dependence between the two measurement
noises, provided their difference satisfies the conditional
moment conditions in Assumption A5.
For fair comparison with the ZO-LMC, we specialize to
independent, mean-zero additive noises with finite variance:
\[
E[\zeta]=0,\qquad
\mathbb E[\zeta^2]=\sigma^2,
\]
with an independent draw of \(\zeta\) at each function evaluation. In particular, within the $j$th summand we write
\[
\widehat f(x+\nu u_{kj}) = f(x+\nu u_{kj})+\zeta_{kj}^{(+)},\qquad
\widehat f(x) = f(x)+\zeta_{kj}^{(0)},
\]
with $\zeta_{kj}^{(+)}$ and $\zeta_{kj}^{(0)}$ independent (and i.i.d.\ across $j$ and $k$).

Roy et al.\ estimate $\nabla f$ via Gaussian smoothing using $b$ i.i.d.\ Gaussian directions
$u_{kj}\sim\mathcal N(0,I_p)$ and a \emph{one-sided} estimator:
\begin{equation}
\label{eq:roy-grad}
\widehat G^{\text{ROY}}_k(x)
=\frac{1}{b}\sum_{j=1}^b
\frac{\widehat f(x+\nu u_{kj})-\widehat f(x)}{\nu}\,u_{kj}.
\end{equation}
In each summand, the two noisy function evaluations carry independent noises $\zeta_{kj}^{(+)}$ and $\zeta_{kj}^{(0)}$.
This baseline uses $2b$ noisy function evaluations per iteration, and $\widehat f(x)$ is resampled at each summand. The batch size $b$ controls the variance reduction
of the estimator under independent oracle noise.

\paragraph{Fair-budget protocol and metrics.}
All methods are compared under the same total number of \emph{oracle calls} (noisy function evaluations $\widehat{f}$) $B$.
Thus the iteration counts are
\[
K_{\text{SP}}=\big\lfloor B/2\big\rfloor,\qquad
K_{\text{FD}}=\big\lfloor B/(2p)\big\rfloor,\qquad
K_{\text{ROY}}=\big\lfloor B/(2b)\big\rfloor,
\]
reflecting the per-iteration oracle costs of SPSA ($2$ calls), FDSA ($2p$ calls), and Roy's estimator ($2b$ calls).
We use the squared error of the post-burn-in empirical mean relative to the
true target mean as our main numerical diagnostic, rather than as a direct
estimate of Wasserstein distance. To assess distributional accuracy beyond the first moment,
we also report first-, second-, and fourth-moment errors as secondary
diagnostics.

\section{LMC-SPSA Algorithm}
We first recall the update rule of the standard Langevin Monte Carlo (LMC). Given a target density $\pi(x)\propto e^{-f(x)}$, LMC constructs a Markov chain $\{X_k\}$ via the recursion:
\begin{equation}\label{eq:ula}
  X_{k+1} \;=\; X_k \;-\; h_k\,\nabla f(X_k) \;+\; \sqrt{2h_k}\,\xi_k,\qquad \xi_k \overset{\mathrm{i.i.d.}}{\sim} N(0,I_p).
\end{equation}
where $h_k>0$ is the stepsize at iteration $k$, and $\xi_k$ is an independent $p$-dimensional standard Gaussian noise. This corresponds to one step of the Euler-Maruyama discretization of the overdamped Langevin diffusion. 
Unadjusted Langevin algorithm (ULA) can mix faster than random-walk MH by leveraging gradient information \cite{sun23}, but a fixed stepsize $h_k\equiv h$ generally induces a stationary bias; shrinking $h_k$ can reduce this bias at the cost of slower per-iteration progress.

\textbf{Gradient Approximation under Noise}: In the LMC-SPSA algorithm, we replace $\nabla f(X_k)$ in \eqref{eq:ula} with the two-point SPSA estimate built from noisy function values $\widehat f$.
Let $c_k>0$. With two noisy evaluations of $f$ at $X_k\pm c_k\Delta_k$, define

\begin{equation}
\widehat{G}_k = \frac{\widehat{f}(X_k + c_k\Delta_k) - \widehat{f}(X_k - c_k\Delta_k)}{2c_k}\Delta_k^{-1},
\end{equation}
\noindent where \(\widehat{f}(X_k \pm c_k\Delta_k) = f(X_k \pm c_k\Delta_k) + \zeta_k^{\pm}\) with $\zeta_k^{\pm}$ representing measurement noise terms that satisfy
\begin{equation}\label{eq:mean0}
\mathbb{E}\!\left[\zeta_k^{(+)} - \zeta_k^{(-)} \,\middle|\, \Delta_k, X_k\right]
= 0 \quad \text{a.s. } \forall k\in\mathbb{N}.
\end{equation}

\noindent where $\Delta_k^{-1}$ denotes the entry-wise inverse of $\Delta_k$. Let us take
$\Delta_k\in\{\pm1\}^p$ i.i.d.\ (Rademacher), so that $\Delta_k^{-1}=\Delta_k$.
Using a Taylor expansion with Lagrangian remainder around $X_k$, we have
\begin{equation*}
\begin{aligned}
\widehat f(X_k \pm c_k\Delta_k)
&= f(X_k) + \zeta_k^{\pm}
\pm c_k\nabla f(X_k)^\top \Delta_k
+ \frac{c_k^2}{2}\Delta_k^\top \nabla^2 f(X_k)\Delta_k \\
&\quad \pm \frac{c_k^3}{6}\nabla^3 f^{(\pm)}\otimes \Delta_k \otimes \Delta_k \otimes \Delta_k,
\end{aligned}
\end{equation*}

\noindent where $\nabla^3 f^{(\pm)}$ is the third-order derivative in the Lagrangian remainder.
Thus, by \eqref{eq:mean0} and Lemma~1 of \cite{spall92}, 

\[
\begin{aligned}
\mathbb E[\widehat G_k\mid X_k]
&=
\nabla f(X_k)
+
\mathbb E\left[
\frac{\zeta_k^{(+)}-\zeta_k^{(-)}}{2c_k}\Delta_k^{-1}
\;\middle|\; X_k
\right]
+
O(c_k^2) \\
&=
\nabla f(X_k)+O(c_k^2),
\end{aligned}
\]

\begin{algorithm}[t]
\caption{LMC-SPSA Sampling Algorithm}
\begin{algorithmic}[1]
\REQUIRE Initial state \(X_0\); stepsize sequence \(\{h_k\}_{k\ge0}\); perturbation sequence \(\{c_k\}_{k\ge0}\); number of iterations \(N\).
\FOR{\(k = 0,\dots,N-1\)}
\STATE Generate \(\Delta_k \in \{\pm1\}^p\) from Rademacher distribution.
\STATE Evaluate noisy log-density: \(\widehat{f}(X_k \pm c_k\Delta_k) = f(X_k \pm c_k\Delta_k) + \zeta_k^{\pm}\).
\STATE \textbf{Compute SPSA gradient estimate:}
\begin{equation*}
\widehat{G}_k = \frac{\widehat{f}(X_k + c_k\Delta_k) - \widehat{f}(X_k - c_k\Delta_k)}{2c_k}\,\Delta_k^{-1}
\end{equation*}
\STATE Update the chain: \(X_{k+1} = X_k - h_k\widehat{G}_k + \sqrt{2h_k}\xi_{k}\)
\ENDFOR
\ENSURE Samples \(\{X_k\}_{k=1}^N\).
\end{algorithmic}
\end{algorithm}

\noindent Therefore, $\widehat{G}_k$ is conditionally nearly unbiased, with bias of order $O(c_k^2)$; in particular, the bias vanishes as $c_k\to 0$.

The pseudocode for the LMC-SPSA algorithm is provided in Algorithm~1. In practice, the sequences $h_k$ and $c_k$ (stepsize and perturbation magnitude) are tuned by the user. In the next section, we provide theoretical conditions on these sequences that ensure convergence of the LMC-SPSA algorithm.

\section[Wasserstein Bounds and Oracle Complexity for a Fixed Stepsize]
{Wasserstein Bounds and Oracle Complexity\\
 for Fixed Stepsize}
In this section, we analyze the convergence of LMC-SPSA in the 2-Wasserstein distance ($W_2$), a standard metric for convergence of stochastic processes. We first derive a finite-sample bound on \(W_2(\nu_K,\pi)\) for fixed stepsize \(h\) and perturbation magnitude \(c\), which characterizes the bias introduced by the fixed stepsize and SPSA approximation. Finally, we present an improved error bound with reduced dependence on the dimension $p$ from $O(p^4)$ in \cite{sun23} to $O(p^2)$.

\subsection{Notation and Assumptions}
In this work, we measure convergence in the quadratic Wasserstein distance $W_2$.
For probability measures $\mu,\nu$ on $\mathbb{R}^p$, the quadratic Wasserstein distance is
\begin{equation}\label{eq:wasserstein2}
W_2(\mu,\nu)
=
\inf_{U\sim\mu,\;V\sim\nu}
\left(\mathbb E\|U-V\|_2^2\right)^{1/2},
\end{equation}
where the infimum is taken over all couplings of \(\mu\) and \(\nu\), i.e., all joint distributions \(\gamma\) of \((U,V)\) with given
marginal distributions \(U\sim\mu\) and \(V\sim\nu\).
This choice aligns with established Langevin analyses: under smooth strong log-concavity and a suitable stepsize, the ULA transition
kernel \(P_h\) is contractive in \(W_2\) \cite{durmus2017}, meaning that
\[
W_2(\mu P_h,\nu P_h)
\le
\rho_h W_2(\mu,\nu)
\quad\text{for some } \rho_h<1.
\]
Let $\nu_k := \mathcal{L}(X_k)$ denote the distribution of $X_k$; our goal is therefore
\begin{equation}\label{eq:W2goal}
W_2(\nu_k,\pi)\longrightarrow 0\quad\text{as }k\to\infty.
\end{equation}
To rigorously analyze convergence, we impose the following assumptions on the log-density \(f(x) = -\log \pi(x)\):

\textbf{Assumption A1} \textit{(Strong Convexity)}:
There exists \(m > 0\) such that for all \( x, y \in \mathbb{R}^p \),
\[
f(x) - f(y) - \nabla f(y)^\top (x - y) \ge \frac{m}{2} \|x - y\|_2^2.
\]

\textbf{Assumption A2} \textit{(Lipschitz Continuity)}:
There exists a constant \(M > 0\) such that for all \( x, y \in \mathbb{R}^p \),
\[
\|\nabla f(x) - \nabla f(y)\|_2 \le M\|x - y\|_2.
\]

\textbf{Assumption A3}
\textit{(Bounded Third-Derivative Operator Norm)}:
The function $f$ belongs to $C^3(\mathbb{R}^p)$, and there exists a constant $M_3>0$, independent of the dimension $p$, such that for all $x\in\mathbb{R}^p$,
\[
\|\nabla^3 f(x)\|_{\mathrm{op}}
:=
\sup_{\|u\|_2=\|v\|_2=\|w\|_2=1}
\left|
\sum_{i,j,\ell=1}^{p}
\frac{\partial^3 f(x)}
{\partial x_i\,\partial x_j\,\partial x_\ell}
u_i v_j w_\ell
\right|
\le M_3.
\]

\textbf{Assumption A4}
\textit{(Uniform $L^2$ Bound on the Gradient)}:
There exists $M_1>0$ such that
\[
\sup_{k\ge 0}\,
\left(\E\big[\|\nabla f(X_k)\|_2^2\big]\right)^{1/2}
\;\le\; M_1.
\]
The constants \(M_1\) and \(M_3\) denote the bounds related to the first and third derivatives of the log-density function \(f\), respectively. Here \(M_1\) may depend on the dimension \(p\), but it is assumed to grow at most linearly, i.e., \(M_1=O(p)\). Hence it does not affect the leading-order dimension dependence in the later bounds.

\textbf{Assumption A5} \textit{(Paired noise conditions).}
    Let \(\mathcal F_k\) denote the history before the two oracle evaluations at
    iteration \(k\). For the two SPSA queries
    \(X_k+c_k\Delta_k\) and \(X_k-c_k\Delta_k\), write
    \[
    \widehat f(X_k+c_k\Delta_k)
    =
    f(X_k+c_k\Delta_k)+\zeta_k^{(+)},
    \qquad
    \widehat f(X_k-c_k\Delta_k)
    =
    f(X_k-c_k\Delta_k)+\zeta_k^{(-)}.
    \]
    We assume that
    \[
    \mathbb E[
    \zeta_k^{(+)}-\zeta_k^{(-)}
    \mid \mathcal F_k,\Delta_k
    ]
    =
    0,
    \]
    and that, for some constant \(\delta\ge 0\),
    \[
    \mathbb E\!\left[
    \left(\zeta_k^{(+)}-\zeta_k^{(-)}\right)^2
    \mid \mathcal F_k,\Delta_k
    \right]
    \le
    2\delta^2.
    \]
    Conditional on $(\mathcal F_k,\Delta_k)$, the paired
measurement-noise difference
$\zeta_k^{(+)}-\zeta_k^{(-)}$
is independent of the Gaussian innovation $\xi_k$.

\emph{Remark on A3.}
The dimension independence of $M_3$ is an explicit requirement
on the family of $f$ under consideration.
An entrywise bound on the third derivatives alone does not
imply this dimension-uniform operator-norm bound.
Examples satisfying A3 are given after Assumption A5.

\emph{Examples satisfying A3.}
For any nondegenerate Gaussian target, the potential is quadratic,
so $\nabla^3 f(x)\equiv 0$ and A3 holds with any fixed $M_3>0$,
independently of $p$.

As a non-Gaussian example, consider
\[
f_p(x)
=
\sum_{i=1}^{p}
\left(\frac{x_i^2}{2}+a\sin x_i\right),
\qquad 0<|a|<1,
\]
where $a$ is fixed independently of $p$.
For any unit vectors $u,v,w\in\mathbb{R}^p$,
\[
\left|\nabla^3 f_p(x)[u,v,w]\right|
\le |a|\sum_{i=1}^{p}|u_i v_i w_i|
\le |a|\|u\|_2\|v\|_2\|w\|_\infty
\le |a|.
\]
Thus, A3 holds with $M_3=|a|$, independently of $p$.
Moreover,
\[
(1-|a|)I_p
\preceq \nabla^2 f_p(x)
\preceq (1+|a|)I_p,
\]
so A1 and A2 also hold with dimension-independent constants
$m=1-|a|$ and $M=1+|a|$.

\emph{Remark on A4.} Verifying A4 is nontrivial. It is a \emph{uniform-in-$k$} bound that depends on the tail behavior of the chain and is typically derived from drift/dissipativity and suitable stepsizes rather than checked directly. For discussion and examples (exponential-family targets and Gaussian mixtures) establishing $L^2$-integrability of the score and its relation to a uniform bound like A4, see \cite{sun23}. In this paper we adopt A4 as a standing condition for our noisy zeroth-order analysis.



\subsection{Bounded Wasserstein Distance with Constant Stepsize}
We bound the three terms in the SPSA gradient error
(Lemma~\ref{lem:spsa-error}) and use these estimates to obtain a
one-step Wasserstein bound (Lemma~\ref{lem:one-step-w2}).
For fixed stepsize \(h\) and perturbation magnitude \(c\), we iterate
this bound to obtain Theorem~\ref{thm:constant_stepsize}.

\begin{lemma}[Bounds for the SPSA gradient error]\label{lem:spsa-error}
Under A3--A5, let \(c>0\) and let the entries of \(\Delta_k\) be
independent Rademacher variables, independent of \(X_k\).
With \(G_k\) denoting the noise-free SPSA estimate, the gradient error
decomposes as
\[
\begin{aligned}
\widehat G(X_k)-\nabla f(X_k)
={}&(\Delta_k^{-1}\Delta_k^\top-I_p)\nabla f(X_k)\\
&+\bigl(G_k-\Delta_k^{-1}\Delta_k^\top\nabla f(X_k)\bigr)\\
&+\frac{\zeta_k^{(+)}-\zeta_k^{(-)}}{2c}\Delta_k^{-1}.
\end{aligned}
\]
The three terms satisfy
\[
\begin{aligned}
\big\|(\Delta_k^{-1}\Delta_k^\top-I_p)\nabla f(X_k)\big\|_{L_2}
 &\le \sqrt{p-1}\,M_1,\\
\big\|G_k-\Delta_k^{-1}\Delta_k^\top\nabla f(X_k)\big\|_{L_2}
 &\le c^2M_3p^2/6,\\
\mathbb E\!\left[\left\|
 \frac{\zeta_k^{(+)}-\zeta_k^{(-)}}{2c}
 \Delta_k^{-1}\right\|^2\right]
 &\le \frac{p\delta^2}{2c^2}.
\end{aligned}
\]
The projection and measurement-noise terms have conditional mean zero
given \(X_k\).
\end{lemma}

\begin{proof}

\noindent For SPSA, with 
$c>0$,
\begin{equation}\label{eq:Gtilde-def}
\widehat G(X_k)
= \frac{\widehat f(X_k + c\Delta_k)-\widehat f(X_k - c\Delta_k)}{2c}\;\Delta_k^{-1}.
\end{equation}
Using $\widehat f(\cdot)=f(\cdot)+\zeta(\cdot)$,
\begin{equation}
\widehat G(X_k)
= G_k \;+\; \frac{\zeta_k^{(+)}-\zeta_k^{(-)}}{2c}\,\Delta_k^{-1}.
\end{equation}
where $G_k$ is the gradient approximation under noise-free settings.
Therefore, by (17) in \cite{sun23}, where \(\Delta_k^{-1}:=(\Delta_{k1}^{-1},\ldots,\Delta_{kp}^{-1})^\top\), and
\(\nabla^3 f(x)\) is understood as the third-order derivative tensor written in
vectorized form, so that for any \(u,v,w\in\mathbb{R}^p\),
\[
\nabla^3 f(x)\,(u\otimes v\otimes w)
:=
\sum_{i_1,i_2,i_3=1}^p
\frac{\partial^3 f(x)}
{\partial x_{i_1}\partial x_{i_2}\partial x_{i_3}}
\,u_{i_1}v_{i_2}w_{i_3}.
\]
Fix an iteration $k$ and define the scalar function
$\phi_k(t)=f(X_k+t\Delta_k)$.
Applying the second-order Taylor formula with a third-order
Lagrange remainder at $t=0$, separately for $t=c$ and $t=-c$,
gives points \[
X_k^+ \in [\,X_k,\; X_k+c\Delta_k\,],
\qquad
X_k^- \in [\,X_k-c\Delta_k,\; X_k\,],
\]such that
\[
f(X_k\pm c\Delta_k)
=
f(X_k)
\pm c\nabla f(X_k)^\top\Delta_k
+\frac{c^2}{2}\Delta_k^\top\nabla^2 f(X_k)\Delta_k
\pm\frac{c^3}{6}
\nabla^3 f(X_k^\pm)[\Delta_k,\Delta_k,\Delta_k].
\]
These are scalar function expansions, so the same two
remainder evaluation points can be used for all coordinates
of the SPSA estimator.

\noindent Subtracting the two expansions and substituting into the
SPSA estimator yields
\begin{equation}
\begin{aligned}
\widehat G(X_k)-\nabla f(X_k)
={}&
\underbrace{
(\Delta_k^{-1}\Delta_k^\top-I_p)\nabla f(X_k)
}_{\text{SP projection error}}
\\
&+
\underbrace{
\frac{c^2}{12}
\Bigl(
\begin{aligned}
&\nabla^3 f(X_k^+)
(\Delta_k\otimes\Delta_k\otimes\Delta_k)
\\
&\quad+
\nabla^3 f(X_k^-)
(\Delta_k\otimes\Delta_k\otimes\Delta_k)
\end{aligned}
\Bigr)\Delta_k^{-1}
}_{\text{third-derivative remainder}}
\\
&+
\underbrace{
\frac{\zeta_k^{(+)}-\zeta_k^{(-)}}{2c}
\Delta_k^{-1}
}_{\text{noise term}}.
\end{aligned}
\label{eq:decomp-3}
\end{equation}

For the first term on the right-hand side of
\eqref{eq:decomp-3}, by (A4), the independence of $\Delta_k$ and $X_k$, and the
fact that $\Delta_k\in\{\pm1\}^p$ has i.i.d.\ Rademacher entries, we
have $\Delta_k^{-1}=\Delta_k$ almost surely and
\[
\mathbb{E}_{\Delta_k}
\left[\Delta_k\Delta_k^\top\right]
=I_p.
\]
Consequently, conditional on $X_k$,
\[
\mathbb{E}_{\Delta_k}
\left[
  \left(
    \Delta_k^{-1}\Delta_k^\top-I_p
  \right)\nabla f(X_k)
  \,\middle|\,
  X_k
\right]
=0.
\]
Moreover, since $\Delta_k^\top\Delta_k=p$ almost surely,
we have
\begin{align*}
\left(
  \Delta_k\Delta_k^\top-I_p
\right)^2
&=
\Delta_k\Delta_k^\top
\Delta_k\Delta_k^\top
-2\Delta_k\Delta_k^\top+I_p
\\
&=
\left(
  \Delta_k^\top\Delta_k
\right)
\Delta_k\Delta_k^\top
-2\Delta_k\Delta_k^\top+I_p
\\
&=
(p-2)\Delta_k\Delta_k^\top+I_p.
\end{align*}
Taking expectation with respect to $\Delta_k$ gives
\begin{align*}
\mathbb{E}_{\Delta_k}
\left[
  \left(
    \Delta_k\Delta_k^\top-I_p
  \right)^2
\right]
&=
(p-2)
\mathbb{E}_{\Delta_k}
\left[
  \Delta_k\Delta_k^\top
\right]
+I_p
\\
&=
(p-1)I_p.
\end{align*}
Therefore, using the independence of $\Delta_k$ and $X_k$,
\begin{align*}
&\mathbb{E}_{\Delta_k}
\left[
  \left\|
    \left(
      \Delta_k^{-1}\Delta_k^\top-I_p
    \right)\nabla f(X_k)
  \right\|_2^2
  \,\middle|\,
  X_k
\right]
\\
&\quad=
\nabla f(X_k)^\top
\mathbb{E}_{\Delta_k}
\left[
  \left(
    \Delta_k\Delta_k^\top-I_p
  \right)^2
\right]
\nabla f(X_k)
\\
&\quad=
(p-1)\left\|\nabla f(X_k)\right\|_2^2.
\end{align*}
Taking expectation with respect to $X_k$ and applying Assumption A4
yields
\begin{align*}
\left\|
  \left(
    \Delta_k^{-1}\Delta_k^\top-I_p
  \right)\nabla f(X_k)
\right\|_{L_2}^2
&=
(p-1)
\mathbb{E}
\left[
  \left\|\nabla f(X_k)\right\|_2^2
\right]
\\
&\le
(p-1)M_1^2.
\end{align*}
Hence,
\begin{equation}
\left\|
  \left(
    \Delta_k^{-1}\Delta_k^\top-I_p
  \right)\nabla f(X_k)
\right\|_{L_2}
\le
\sqrt{p-1}\,M_1.
\label{eq:sharpened-sp-projection-bound}
\end{equation}

For the second term, subtract the noise term from both sides of
\eqref{eq:decomp-3} and collect the terms containing \(\nabla f(X_k)\).
Since \(G_k\) is the noise-free estimate, this gives
\[
\begin{aligned}
&G_k-\Delta_k^{-1}\Delta_k^\top\nabla f(X_k)\\
&\quad=\frac{c^2}{12}\Bigl(
  \nabla^3 f(X_k^+)(\Delta_k\otimes\Delta_k\otimes\Delta_k)
  +\nabla^3 f(X_k^-)(\Delta_k\otimes\Delta_k\otimes\Delta_k)
  \Bigr)\Delta_k^{-1}.
\end{aligned}
\]
Since \(\Delta_k\in\{\pm1\}^p\) has i.i.d.\ Rademacher coordinates, we have
\[
\|\Delta_k\|_2=\|\Delta_k^{-1}\|_2=\sqrt p
\qquad \text{a.s.}
\]
Hence by A3, we have
\[
\left|
\nabla^3 f(X_k^\pm)\bigl(\Delta_k\otimes\Delta_k\otimes\Delta_k\bigr)
\right|
\le
M_3\|\Delta_k\|_2^3
=
M_3p^{3/2}.
\]
Therefore,
\begin{align}
&\Big\|
\bigl(
\nabla^3 f(X_k^+)\bigl(\Delta_k\otimes\Delta_k\otimes\Delta_k\bigr)
+
\nabla^3 f(X_k^-)\bigl(\Delta_k\otimes\Delta_k\otimes\Delta_k\bigr)
\bigr)\Delta_k^{-1}
\Big\|_{L_2}
\notag\\
&\le \Bigl\|
\Big(
\big|
\nabla^3 f(X_k^+)\bigl(\Delta_k\otimes\Delta_k\otimes\Delta_k\bigr)
\big|
+
\big|
\nabla^3 f(X_k^-)\bigl(\Delta_k\otimes\Delta_k\otimes\Delta_k\bigr)
\big|
\Big)\,\|\Delta_k^{-1}\|_2\Bigr\|_{L_2}
\notag\\
&\le
\bigl(M_3p^{3/2}+M_3p^{3/2}\bigr)\sqrt p
=
2M_3p^2.
\label{eq:bound-third-improved}
\end{align}
Multiplying \eqref{eq:bound-third-improved} by \(c^2/12\) gives
the bound for the full Taylor remainder,
\begin{equation}
\left\|G_k-\Delta_k^{-1}\Delta_k^\top\nabla f(X_k)\right\|_{L_2}
\le \frac{c^2M_3p^2}{6}.
\label{eq:full-taylor-remainder}
\end{equation}

For the measurement-noise term, conditioning first on
\((\mathcal F_k,\Delta_k)\) and using A5 gives
\begin{equation}
\mathbb E\!\left[
\frac{\zeta_k^{(+)}-\zeta_k^{(-)}}{2c}\Delta_k^{-1}
\,\middle|\, X_k\right]=0.
\label{eq:noise-mean}
\end{equation}
Since \(\|\Delta_k^{-1}\|_2^2=p\), the second moment is
\[
\frac{p}{4c^2}\,
\mathbb E\!\left[
 \mathbb E\!\left[(\zeta_k^{(+)}-\zeta_k^{(-)})^2
     \,\middle|\,\mathcal F_k,\Delta_k\right]\right].
\]
By A5 the inner expectation is at most \(2\delta^2\). Hence
\begin{equation}
\mathbb E\!\left[\left\|
 \frac{\zeta_k^{(+)}-\zeta_k^{(-)}}{2c}
 \Delta_k^{-1}\right\|_2^2\right]\le \frac{p\delta^2}{2c^2}.
\label{eq:noise-var}
\end{equation}
\end{proof}

We now apply the estimates in Lemma~\ref{lem:spsa-error} to the
difference between one LMC-SPSA update and the Langevin diffusion
over the same time interval.

\begin{lemma}[One-step Wasserstein bound]\label{lem:one-step-w2}
Under A1--A5, fix \(0<h=h_k\le 2/(m+M)\) and \(c>0\).
Then the laws of two consecutive iterates satisfy
\begin{equation}
\begin{aligned}
W_2^2(\nu_{k+1},\pi)
\le{}&
\left[
 (1-mh)W_2(\nu_k,\pi)
 +\frac{7\sqrt{2}}{6}M(h^3p)^{1/2}
 +\frac{1}{3}h c^2 M_3 p^2
\right]^2\\
&+2(p-1)h^2M_1^2+\frac{h^2p\delta^2}{2c^2}.
\end{aligned}
\label{eq:w2_onestep}
\end{equation}
\end{lemma}

\begin{proof}

\noindent
Fix an iteration $k\ge 0$ and let $h:=h_k$.
Let $X_k\sim \nu_k$ be the current iterate and let $L_0\sim \pi$ be a random vector
defined on the same probability space as $X_k$ (under an arbitrary coupling of $(\nu_k,\pi)$).
Let $\{L_s\}_{0\le s\le h}$ be the Langevin diffusion starting from $L_0$, i.e.,
\begin{equation}\label{eq:langevin_diffusion_integral}
L_h
=
L_0-\int_0^h \nabla f(L_s)\,\mathrm{d}s+\sqrt{2}\,W_h,\qquad h>0,
\end{equation}
where $W_h$ is a $p$-dimensional Brownian increment with $W_h\sim \mathcal{N}(0,hI_p)$.
Equivalently, we may write $W_h=\sqrt{h}\,\xi_k$ with $\xi_k\sim\mathcal{N}(0,I_p)$.

\noindent Let
\[
D_{0k}:=L_0-X_k,\qquad D_{h,k+1}:=L_h-X_{k+1}.
\]

\noindent The local deviation
\begin{equation}\label{eq:D-increment}
D_{h,k+1}-D_{0k} \;=\; L_h - X_{k+1} - L_0 + X_k ,
\end{equation}
where $L_h$ is the Langevin diffusion at time $h$. Then
\begin{align}
D_{h,k+1}-D_{0k}
&= -\int_{0}^{h}\!\nabla f(L_s)\,ds \;+\; \sqrt{2h}\,\xi_k
   \;+\; h\,\widehat G(X_k) \;-\; \sqrt{2h_k}\,\xi_k \notag\\
&= -\int_{0}^{h}\!\nabla f(L_s)\,ds \;+\; h\,\widehat G(X_k). \label{eq:Wh-compact}
\end{align}

\noindent Hence
\begin{align}
D_{h,k+1}
&= D_{0k}
 - h\!\left[\nabla f(L_0)-\nabla f(X_k)\right]
 + h\!\left(\widehat G(X_k)-\nabla f(X_k)\right) \notag\\
&\quad - \int_{0}^{h}\!\big(\nabla f(L_s)-\nabla f(L_0)\big)\,ds .
\label{eq:D-recursion}
\end{align}
Substitute the decomposition \eqref{eq:decomp-3} from
Lemma~\ref{lem:spsa-error} for \(\widehat G(X_k)-\nabla f(X_k)\)
in \eqref{eq:D-recursion}.

\noindent The projection error and the Taylor remainder depend on the same
perturbation \(\Delta_k\), so their cross term must be controlled. By the Cauchy--Schwarz inequality, followed by
Young's inequality $2ab\le a^2+b^2$, we obtain
\begin{align*}
&2h^2\left|
\mathbb{E}\left\langle
(\Delta_k^{-1}\Delta_k^\top-I_p)\nabla f(X_k),
G_k-\Delta_k^{-1}\Delta_k^\top\nabla f(X_k)
\right\rangle
\right|
\\
&\quad\le
2h^2
\left\|
(\Delta_k^{-1}\Delta_k^\top-I_p)\nabla f(X_k)
\right\|_{L_2}
\left\|
G_k-\Delta_k^{-1}\Delta_k^\top\nabla f(X_k)
\right\|_{L_2}
\\
&\quad\le
h^2
\left\|
(\Delta_k^{-1}\Delta_k^\top-I_p)\nabla f(X_k)
\right\|_{L_2}^2
+
h^2
\left\|
G_k-\Delta_k^{-1}\Delta_k^\top\nabla f(X_k)
\right\|_{L_2}^2
\\
&\quad\le
(p-1)h^2M_1^2
+\frac{h^2c^4M_3^2p^4}{36}.
\end{align*}

\noindent The last inequality uses the projection bound
\eqref{eq:sharpened-sp-projection-bound} and the Taylor-remainder
bound \eqref{eq:full-taylor-remainder} from Lemma~\ref{lem:spsa-error}.
Squaring these bounds and multiplying by \(h^2\) gives the two terms
in the last line. 

The additional squared remainder $h^2c^4M_3^2p^4/36=b^2$,
where $b=hc^2M_3p^2/6$, is absorbed using
$(a+b)^2+b^2\le(a+2b)^2$ for $a,b\ge0$.
Here $a$ collects the $L_2$ bounds for the contraction and
discretization terms, so the remainder coefficient inside
the square increases from $1/6$ to $1/3$. The cross-term estimate also contributes a second copy of the
projection variance bound. For the measurement-noise term,
\eqref{eq:noise-var} of Lemma~\ref{lem:spsa-error}, multiplied by
\(h^2\), contributes \(h^2p\delta^2/(2c^2)\).
Combining these estimates gives

\begin{align}
\|D_{h,k+1}\|_{L_2}^2
&\le
\Biggl(
\left\|
D_{0k}
-
h\bigl(\nabla f(X_{k}+D_{0k})-\nabla f(X_{k})\bigr)
\right\|_{L_2}
+\frac{1}{3}h c^2 M_3 p^2
\notag\\
&\qquad\qquad
+\left\|
\int_0^{h}\bigl(\nabla f(L_t)-\nabla f(L_0)\bigr)\,dt
\right\|_{L_2}
\Biggr)^2
+2(p-1)h^2M_1^2
+\frac{h^2 p\delta^2}{2c^2}.
\label{eq:Dkh_bound}
\end{align}

\noindent Under A1--A2 and $0<h\le 2/(m+M)$, the deterministic
gradient step is contractive with factor $\rho_k=1-mh$.
Consequently, by the gradient-step contraction bound of
Dalalyan and Karagulyan~\cite[Lemma~2]{dalalyan19},
\begin{equation}
\left\|
D_{0k}
-
h\bigl(\nabla f(X_{k}+D_{0k})-\nabla f(X_{k})\bigr)
\right\|_{L_2}
\le
\rho_k \|D_{0k}\|_{L_2}.
\label{eq:lemma2_bound}
\end{equation}

\noindent The discretization error is bounded using the Lipschitz
continuity of $\nabla f$ and the stationarity of the
Langevin diffusion, as quantified in Lemma 4 of \cite{dalalyan19},
\begin{equation}
\left\|
\int_0^{h}\bigl(\nabla f(L_t)-\nabla f(L_0)\bigr)\,dt
\right\|_{L_2}
\le
\frac12\bigl(h^4M^3p\bigr)^{1/2}
+\frac23(2h^3p)^{1/2}M
\le
\frac{7\sqrt2}{6}(h^3p)^{1/2}M.
\label{eq:lemma4_bound}
\end{equation}

\noindent Choose $(X_k,L_0)$ as an optimal coupling of $\nu_k$ and $\pi$.
Since the diffusion is stationary, both $L_0$ and $L_h$
have distribution $\pi$. Consequently,
\begin{equation}
W_2^2(\nu_k,\pi)
=
\|D_{0k}\|_{L_2}^2,
\qquad
W_2^2(\nu_{k+1},\pi)
\le
\|D_{h,k+1}\|_{L_2}^2.
\label{eq:w2_coupling}
\end{equation}
The second relation is an inequality because the coupling $(X_{k+1},L_h)$ induced by the one-step evolution need not be optimal. Here $D_{h,k+1}$ refers to the discrepancy under the coupling obtained by evolving the current pair $(L_0,X_k)$ over one step. The subscript $h$ records the elapsed diffusion time. Resetting this time to zero for the next step does not make the resulting coupling optimal.

Then combining \eqref{eq:Dkh_bound}, \eqref{eq:lemma2_bound}, \eqref{eq:lemma4_bound}, and \eqref{eq:w2_coupling} and substituting \(\rho_k=1-mh\), we obtain \eqref{eq:w2_onestep}.

\end{proof}

For constant \(h\) and \(c\), iterating Lemma~\ref{lem:one-step-w2}
gives the following finite-sample bound.

\begin{theorem}[Finite-sample bound for noisy LMC-SPSA with constant stepsize]\label{thm:constant_stepsize}
\label{thm:improved_bound}

Let $f:\mathbb{R}^p \to \mathbb{R}$ satisfy assumptions A1--A5. Consider the LMC-SPSA algorithm with constant stepsize $h_k \equiv h$ and SPSA perturbation magnitude $c_k \equiv c$. Suppose $h$ is chosen such that $0 < h \le 2/({m+M})$. Then for all $K \ge 1$, the Wasserstein-2 distance between $\nu_K$ (the law of $X_K$) and the target $\pi$ is bounded by 
\begin{equation}\label{eq:W2_bound_clean}
\begin{split}
W_2(\nu_K,\pi)
\le\;&
(1-mh)^K\,W_2(\nu_0,\pi)
+\frac{7\sqrt2\,M}{6m}\sqrt{hp}
+\frac{M_3}{3m}c^2p^2 \\
&+\sqrt{\frac{2(p-1)h}{m}}\,M_1
+\sqrt{\frac{hp}{2m}}\frac{\delta}{c}.
\end{split}
\end{equation}
\end{theorem}

Under the dimension-uniform bound on \(\|\nabla^3 f\|_{\mathrm{op}}\)
in A3, the Taylor remainder contributes \(M_3c^2p^2/(3m)\) to
\eqref{eq:W2_bound_clean}. The corresponding remainder estimate in
Sun and Spall~\cite{sun23} scales as \(c^2p^4\).

\begin{proof}
\begingroup
\setlength{\abovedisplayskip}{4pt}
\setlength{\belowdisplayskip}{4pt}
\setlength{\jot}{2pt}

\noindent With \(h_k\equiv h\) and \(c_k\equiv c\),
Lemma~\ref{lem:one-step-w2} gives the same coefficients in
\eqref{eq:w2_onestep} at every iteration. Thus we can write
\[
W_2^2(\nu_{k+1},\pi)
\le \bigl((1-A)W_2(\nu_k,\pi)+C\bigr)^2+B_{\mathrm{new}}^2,
\]
where
\[
\begin{aligned}
A &= mh,\\
B_{\mathrm{new}}^2
&=2(p-1)h^2M_1^2+\frac{h^2p\delta^2}{2c^2},\\
C
&=\frac{7\sqrt{2}}{6}M(h^3p)^{1/2}
  +\frac13 hc^2M_3p^2.
\end{aligned}
\]

\noindent The expression for $B$ printed in the proof of Theorem 1
in \cite{sun23} contains a typographical error. The final
$M_1$-dependent term there corresponds to
$B^2=p(p-1)h^2M_1^2$.
\begin{align}
\intertext{Applying ~\cite[Lemma~1]{sun23}
yields}
W_2\!\left(\nu_{k},\pi\right)
&\le
(1-mh)^{k}\,W_2\!\left(\nu_0,\pi\right)
+ \frac{C}{A}
+ \frac{B_{\text{new}}^{\,2}}{\,C+\sqrt{A\,B_{\text{new}}^{\,2}}\,}.
\label{eq:W2-final-short}
\end{align}

\endgroup

\noindent Since \(C>0\), the fractional term vanishes when
\(B_{\mathrm{new}}=0\). For \(B_{\mathrm{new}}>0\), we have
\begin{align}
\frac{B_{\mathrm{new}}^{\,2}}{C+\sqrt{A B_{\mathrm{new}}^{\,2}}}
&\le
\frac{B_{\mathrm{new}}^{\,2}}{\sqrt{A B_{\mathrm{new}}^{\,2}}}
=
\frac{B_{\mathrm{new}}}{\sqrt A}
=
\sqrt{\frac{2(p-1)}{m}\,h M_1^2+\frac{hp\delta^2}{2m c^2}}
\notag\\
&\le
\sqrt{\frac{2(p-1)}{m}}\,M_1\sqrt{h}
+
\sqrt{\frac{hp}{2m}}\frac{\delta}{c}.
\label{eq:Bnew_simplify}
\end{align}

Taking \(k=K\) and substituting \(C/A\) gives \eqref{eq:W2_bound_clean}.

Verification that the noisy upper bound dominates the noise-free upper bound.
For this comparison, let \(B^2=2(p-1)h^2M_1^2\) denote the value of \(B_{\mathrm{new}}^2\) at \(\delta=0\). The two bounds differ only through the terms
\[
\frac{B^2}{C+\sqrt{AB^2}}
\qquad\text{and}\qquad
\frac{B_{\mathrm{new}}^2}
     {C+\sqrt{AB_{\mathrm{new}}^2}}.
\]
For $\delta>0$, we have $B_{\mathrm{new}}^2>B^2$, so it suffices to show
that this expression is increasing in the squared
$B$-parameter. To this end, define
\[
\phi(x)=\frac{x}{C+\sqrt{Ax}},\qquad x\ge0.
\]
For $x>0$,
\[
\phi'(x)
=
\frac{C+\frac12\sqrt{Ax}}
     {(C+\sqrt{Ax})^2}
>0.
\]
Together with continuity at $x=0$, this shows that
$\phi$ is increasing on $[0,\infty)$.
Hence $\phi(B_{\mathrm{new}}^2)>\phi(B^2)$

Therefore, the upper bound for
\(W_2(\nu_k,\pi)\) in the noisy setting is larger than the corresponding upper bound in the
noise-free setting when \(\delta>0\), with equality at \(\delta=0\). This completes the verification.
\end{proof}

\paragraph{Remark.}
For later sample-complexity and parameter-tuning arguments, we record
the corresponding unsimplified sharpened bound below. This bound is
tighter than the simplified display bound and preserves the structure
needed for the explicit balancing of the error terms:
\begin{equation}
\begin{aligned}
W_2(\nu_K,\pi)
\le{}&
(1-mh)^K W_2(\nu_0,\pi)
+\frac{C}{A}
+\frac{B_{\mathrm{new}}^2}
{C+\sqrt{A B_{\mathrm{new}}^2}},
\end{aligned}
\label{eq:sharpened-constant-bound}
\end{equation}
where
\begin{align}
A
&:=mh,\\
B_{\mathrm{new}}^2
&:=
2(p-1)h^2M_1^2
+\frac{h^2p\delta^2}{2c^2},\\
C
&:=
\frac{7\sqrt{2}}{6}M(h^3p)^{1/2}
+\frac{1}{3}hc^2M_3p^2.
\end{align}
We next use this bound to choose fixed values of \(h\) and \(c\) and
derive the oracle complexity, measured by the number of function
evaluations sufficient to achieve \(W_2(\nu_N,\pi)\le\epsilon\).

\subsection{Oracle Complexity and Comparison with Roy et al.}
\label{sec:comparison_roy}

We compare LMC-SPSA with the smoothing-based zeroth-order Langevin
methods of Roy et al.~\cite{Roy2022}. For a target accuracy
\(W_2(\nu_N,\pi)\le\epsilon\), we count each function evaluation as
one oracle call. LMC-SPSA uses two evaluations per iteration.
The methods of Roy et al. average two-point differences over a batch
of Gaussian directions, so their cost per iteration depends on the
batch size.

\paragraph{Oracle models.}
Both methods only query noisy function values (zeroth-order information). In the shared-randomness model
of Roy et al.~\cite{Roy2022}, the two evaluations in each difference
use the same random input \(\xi\). Their analysis assumes
\[
  \mathbb E[F(\theta,\xi)]=f(\theta),\qquad
  \mathbb E[\nabla F(\theta,\xi)]=\nabla f(\theta),\qquad
  \mathbb E\bigl[\|\nabla F(\theta,\xi)-\nabla f(\theta)\|^2\bigr]
  \le \sigma^2.
\]
Roy et al. also analyze independent additive noise at the two query
points. Our analysis does not require the two measurement noises to be independent. For LMC-SPSA, Assumption~A5 places conditions directly on
the difference between the two measurement errors. This difference
has conditional mean zero and uniformly bounded conditional second
moment and is conditionally independent of the Gaussian innovation.
The two measurement errors may therefore be correlated.

We first use the upper bound of \(W_2\) from Theorem~\ref{thm:constant_stepsize} to choose fixed
parameters \(c,h\) and an iteration count \(N\) for the prescribed
accuracy $\epsilon$. We then compare the resulting oracle bounds under
independent additive noise, a setting covered by both analyses.

\begin{theorem}[Parameter choices for a target $W_2$ accuracy]
\label{thm:corrected-explicit}
Suppose that the conditions of Theorem~\ref{thm:constant_stepsize} hold.
Given a target $W_2$ accuracy $\epsilon\in(0,1)$,
choose the perturbation magnitude $c$, step size $h$,
and iteration count $N$ as follows:
\begin{align}
c^2
&:=
\frac{3m\epsilon}{4M_3p^2},
\label{eq:explicit-c}\\
h
&:=
\min\left\{
\frac{2}{m+M},
\frac{9m^2\epsilon^2}{392M^2p},
\frac{m\epsilon^2}
{32\left(
(p-1)M_1^2+
\dfrac{\delta^2M_3p^3}{3m\epsilon}
\right)}
\right\},
\label{eq:explicit-h}\\
N
&:=
\left\lceil
\frac{1}{mh}
\log\left(
\max\left\{
1,\frac{4W_2(\nu_0,\pi)}{\epsilon}
\right\}
\right)
\right\rceil.
\label{eq:explicit-N}
\end{align}
If the denominator of the last stepsize constraint vanishes, that
constraint is omitted. Then $W_2(\nu_N,\pi)\le\epsilon$. Since every LMC-SPSA
iteration uses exactly two function evaluations, we know
$\operatorname{OracleCalls}=2N$. For a family with \(m\) bounded away
from zero and \(M,M_1,M_3\) uniformly bounded in the dimension and
target accuracy, the resulting bounds simplify to
\begin{align}
\operatorname{OracleCalls}
&=
\widetilde{\mathcal O}\left(
\frac{p}{\epsilon^2}
+
\frac{\delta^2p^3}{\epsilon^3}
\right),
\label{eq:corrected-complexity-general}\\
\operatorname{OracleCalls}
&=
\widetilde{\mathcal O}\left(
\frac{p}{\epsilon^2}
\right), when \ \delta=0,
\label{eq:corrected-complexity-zero}
\end{align}
where \(\widetilde{\mathcal O}\) suppresses logarithmic factors,
including the initialization factor in \eqref{eq:explicit-N}.
\end{theorem}
\begin{proof}
If \(N=0\), then \eqref{eq:explicit-N} gives
\(W_2(\nu_0,\pi)\le\epsilon/4\). For \(N\ge1\), use the
unsimplified bound \eqref{eq:sharpened-constant-bound} from
Theorem~\ref{thm:constant_stepsize} and expand \(C/A\) into its
discretization and Taylor-remainder terms. We allocate \(\epsilon/4\)
to each of the resulting four terms. First, since $m\le M$ and $h\le 2/(m+M)$, we have $0<mh\le 1$.
By the choice of $N$ in \eqref{eq:explicit-N},
\[
  mhN \ge
  \log\!\left(\max\!\left\{1,\frac{4W_2(\nu_0,\pi)}{\epsilon}\right\}\right).
\]
Using $1-mh\le e^{-mh}$, we obtain
\[
  (1-mh)^N W_2(\nu_0,\pi)
  \le e^{-mhN} W_2(\nu_0,\pi)
  \le \frac{W_2(\nu_0,\pi)}
           {\max\!\left\{1,4W_2(\nu_0,\pi)/\epsilon\right\}}
  \le \frac{\epsilon}{4}.
\]

Second, the second constraint in \eqref{eq:explicit-h} gives
\[
  \frac{7\sqrt{2}M}{6m}\sqrt{hp}
  \le \frac{\epsilon}{4}.
\]
Third, \eqref{eq:explicit-c} gives
\[
  \frac{M_3}{3m}c^2p^2
  =\frac{\epsilon}{4}.
\]
Finally,
\begin{align}
\frac{B_{\mathrm{new}}^2}
{C+\sqrt{A B_{\mathrm{new}}^2}}
&\le
\frac{B_{\mathrm{new}}}{\sqrt{A}}
\nonumber\\
&=
\sqrt{
\frac{h}{m}
\left(
2(p-1)M_1^2
+\frac{p\delta^2}{2c^2}
\right)
}
\nonumber\\
&=
\sqrt{
\frac{h}{m}
\left(
2(p-1)M_1^2
+\frac{2\delta^2M_3p^3}{3m\epsilon}
\right)
}
\le
\frac{\epsilon}{4},
\label{eq:fractional-term-bound}
\end{align}
where the final expression is zero if \(p=1\) and \(\delta=0\);
otherwise the inequality follows from the third constraint in
\eqref{eq:explicit-h}.  Adding the four bounds proves the stated
accuracy guarantee.
To bound the number of function evaluations, we first take
reciprocals in the definition of $h$ in \eqref{eq:explicit-h}.
This gives
\[
\frac{1}{mh}
=
\max\left\{
  \frac{m+M}{2m},\,
  \frac{392M^2p}{9m^3\epsilon^2},\,
  \frac{32(p-1)M_1^2}{m^2\epsilon^2}
  +\frac{32\delta^2M_3p^3}{3m^3\epsilon^3}
\right\}.
\]
Since $p\ge 1$ and $0<\epsilon<1$, we have
$1\le p/\epsilon^2$ and $(p-1)/\epsilon^2\le p/\epsilon^2$.
Thus, suppressing the problem-dependent constants
$m,M,M_1,M_3$, we obtain
\[
\frac{1}{mh}
=
\mathcal O\left(
  \frac{p}{\epsilon^2}
  +\frac{\delta^2p^3}{\epsilon^3}
\right).
\]

Each iteration uses two function evaluations. Hence,
by \eqref{eq:explicit-N} and $\lceil x\rceil\le x+1$,
\[
\operatorname{OracleCalls}
=2N
\le
2+\frac{2}{mh}
\log\left(\max\left\{1,\frac{4W_2(\nu_0,\pi)}{\epsilon}\right\}\right).
\]
Substituting the bound on $1/(mh)$ and suppressing the logarithmic
factor gives
\[
\operatorname{OracleCalls}
=
\widetilde{\mathcal O}\left(
  \frac{p}{\epsilon^2}
  +\frac{\delta^2p^3}{\epsilon^3}
\right),
\]
which proves \eqref{eq:corrected-complexity-general}.
Setting $\delta=0$ eliminates the second term and gives
\eqref{eq:corrected-complexity-zero}.
\end{proof}

\paragraph{Comparison under independent additive noise.}
\label{rem:roy-comparison}
For a direct oracle comparison, consider the independent, mean-zero
additive errors of Roy et al.~\cite{Roy2022}, each with variance
\(\sigma^2\). Their difference has conditional second moment \(2\sigma^2\),
so we may take \(\delta^2=\sigma^2\) in A5.
Table~\ref{tab:corrected-complexity} compares the resulting upper bounds
under the respective target and initialization assumptions. For Roy et al., a batch of size \(b\)
uses \(2b\) function evaluations.

Under the uniform bounds in Theorem~\ref{thm:corrected-explicit}, the
noise-dependent term for LMC-SPSA scales as
\(\widetilde{\mathcal O}(\sigma^2p^3/\epsilon^3)\).
For both ZO-LMC and ZO-KLMC, Roy et al.~\cite{Roy2022} give a total
of order \(\widetilde{\mathcal O}(\max\{1,\sigma^2\}p^2/\epsilon^4)\)
under independent additive noise. Thus the noise-dependent LMC-SPSA
bound has a smaller power of \(1/\epsilon\) and a larger power of \(p\).
When \(\sigma=\Theta(1)\), the ratio of the displayed leading terms
scales as \(p\epsilon\), up to constants and logarithms.

\paragraph{Paired evaluations and shared randomness.}
If the two SPSA evaluations share a realization of a query-independent
additive error, then \(\zeta_k^{(+)}=\zeta_k^{(-)}\)
almost surely. Hence \(\delta=0\), and
\eqref{eq:corrected-complexity-zero} gives
\(\widetilde{\mathcal O}(p/\epsilon^2)\) function evaluations,
as shown in the last row of Table~\ref{tab:corrected-complexity}.
More generally, even with a common random input, the measurement
error can depend on the query point, in which case the two errors
can differ.

For ZO-LMC under their shared-randomness assumptions,
Roy et al.~\cite{Roy2022} use a batch of order
\(\max\{1,\sigma^2\}p\) and
\(\widetilde{\mathcal O}(p/\epsilon^2)\) iterations, giving
\(\widetilde{\mathcal O}(\max\{1,\sigma^2\}p^2/\epsilon^2)\)
function evaluations. Here \(\sigma^2\) bounds the mean squared
gradient error. In the independent additive-noise comparison,
\(\sigma^2\) denotes the variance of a function measurement.

\begin{table}[t]
\centering
\caption{Oracle-complexity comparison for
$W_2(\nu_N,\pi)\le\epsilon$, where
$\kappa_\sigma:=\max\{1,\sigma^2\}$ and
$\widetilde{\mathcal O}(\cdot)$ suppresses logarithmic factors. The LMC-SPSA rows give the simplified
orders under the uniform bounds stated in Theorem~\ref{thm:corrected-explicit}.}
\label{tab:corrected-complexity}

\renewcommand{\arraystretch}{1.25}
\renewcommand{\cellalign}{cc}
\setlength{\tabcolsep}{3pt}
\footnotesize

\begin{tabular}{
@{}
>{\raggedright\arraybackslash}m{0.17\textwidth}
>{\raggedright\arraybackslash}m{0.17\textwidth}
>{\centering\arraybackslash}m{0.22\textwidth}
>{\centering\arraybackslash}m{0.17\textwidth}
>{\centering\arraybackslash}m{0.19\textwidth}
@{}
}
\toprule

Oracle regime
&
Method
&
Calls per iteration
&
Iterations $N$
&
Total calls
\\

\midrule

\makecell[l]{
Independent\\
additive noise
}
&
\makecell[l]{
Roy et al.\\
ZO-LMC
}
&
\makecell[c]{
$2b=\widetilde{\mathcal O}
\!\left(
\dfrac{\kappa_\sigma p}{\epsilon^2}
\right)$
}
&
$\widetilde{\mathcal O}
\!\left(
\dfrac{p}{\epsilon^2}
\right)$
&
$\widetilde{\mathcal O}
\!\left(
\dfrac{\kappa_\sigma p^2}{\epsilon^4}
\right)$
\\[6pt]

\addlinespace[3pt]

\makecell[l]{
Independent\\
additive noise
}
&
\makecell[l]{
Roy et al.\\
ZO-KLMC
}
&
\makecell[c]{
$2b=\widetilde{\mathcal O}
\!\left(
\dfrac{\kappa_\sigma p^{3/2}}{\epsilon^3}
\right)$
}
&
$\widetilde{\mathcal O}
\!\left(
\dfrac{\sqrt p}{\epsilon}
\right)$
&
$\widetilde{\mathcal O}
\!\left(
\dfrac{\kappa_\sigma p^2}{\epsilon^4}
\right)$
\\[6pt]

\addlinespace[3pt]

\makecell[l]{
Independent\\
additive noise
}
&
\makecell[l]{
This work:\\
LMC-SPSA\\
($\delta^2$ = $\sigma^2$)
}
&
$2$
&
$\widetilde{\mathcal O}
\!\left(
\dfrac{p}{\epsilon^2}
+
\dfrac{\sigma^2p^3}{\epsilon^3}
\right)$
&
$\widetilde{\mathcal O}
\!\left(
\dfrac{p}{\epsilon^2}
+
\dfrac{\sigma^2p^3}{\epsilon^3}
\right)$
\\[6pt]

\addlinespace[3pt]

\makecell[l]{
Perfectly paired\\
additive noise
}
&
\makecell[l]{
This work:\\
LMC-SPSA\\
$\delta=0$
}
&
$2$
&
$\widetilde{\mathcal O}
\!\left(
\dfrac{p}{\epsilon^2}
\right)$
&
$\widetilde{\mathcal O}
\!\left(
\dfrac{p}{\epsilon^2}
\right)$
\\[4pt]

\bottomrule
\end{tabular}
\end{table}

\FloatBarrier

\section{Convergence under Shrinking Schedules}
\label{subsec:diminishing-schedule}

Section~\ref{sec:comparison_roy} uses fixed parameters \(h,c\)
chosen for a prescribed accuracy. We now ask when a single run of
LMC-SPSA yields \(W_2(\nu_k,\pi)\to0\) as both parameters \(h,c\) decrease to zero. In the constant-parameter bound \eqref{eq:W2_bound_clean} of
Theorem~\ref{thm:constant_stepsize}, a smaller \(c\) reduces the
Taylor remainder, while the measurement-noise term is proportional
to \(\delta\sqrt{h}/c\). This tradeoff motivates the joint choice
of \(h_k\) and \(c_k\).

We first consider the contraction of the initial error for a
power-law stepsize.

\begin{lemma}[Decay of the contraction product for power-law stepsizes]
\label{lem:contraction-product}
Let $0<m\le M$ and
\[
h_k=\min\left\{
\frac{2}{m+M},\frac{1}{(k+1)^\alpha}
\right\},
\qquad k\ge0,
\]
where $\alpha\in(0,1]$. Then
\[
\prod_{j=0}^{k-1}(1-mh_j)\longrightarrow0
\qquad\text{as }k\to\infty.
\]
\end{lemma}
\begin{proof}
Since $0<m\le M$ and $0<h_j\le 2/(m+M)$, we have
\[
0<mh_j\le\frac{2m}{m+M}\le1.
\]
Therefore, using $1-x\le e^{-x}$,
\[
0\le
\prod_{j=0}^{k-1}(1-mh_j)
\le
\exp\left(-m\sum_{j=0}^{k-1}h_j\right).
\]
Since $(j+1)^{-\alpha}\to0$, there exists $k_0$ such that
$h_j=(j+1)^{-\alpha}$ for all $j\ge k_0$. Hence
\[
\sum_{j=0}^{k-1}h_j
\ge
\sum_{j=k_0}^{k-1}(j+1)^{-\alpha}
\longrightarrow\infty,
\]
because $\alpha\in(0,1]$.
The exponential upper bound therefore tends to zero,
and the conclusion follows.
\end{proof}

Lemma~\ref{lem:contraction-product} gives a family of stepsizes
for which the initial error vanishes. For the general sequences
below, the same exponential estimate applies whenever
\(0<h_k\le 2/(m+M)\) and \(\sum_k h_k=\infty\).
We now use Lemma~\ref{lem:one-step-w2} to control the errors
accumulated along the chain. This extends the constant-stepsize
analysis of Sun and Spall~\cite{sun23} to shrinking schedules.

\begin{theorem}[Wasserstein convergence under a shrinking schedule]
\label{thm:shrinking-noisy}
Suppose that Assumptions A1--A5 hold, with an arbitrary fixed
$\delta\ge0$. Consider Algorithm~1 with positive sequences
$\{h_k\}_{k\ge0}$ and $\{c_k\}_{k\ge0}$ satisfying
\begin{equation}
0<h_k\le\frac{2}{m+M},
\qquad
\sum_{k=0}^{\infty}h_k=\infty,
\qquad
h_k\longrightarrow0,
\label{eq:shrinking-h-conditions}
\end{equation}
and
\begin{equation}
c_k\longrightarrow0,
\qquad
\frac{h_k}{c_k^2}\longrightarrow0.
\label{eq:shrinking-c-conditions}
\end{equation}
Define
\begin{align}
C_k
&:=
\frac{7\sqrt{2}}{6}M(h_k^3p)^{1/2}
+
\frac{1}{3}h_kc_k^2M_3p^2,
\label{eq:noisy-shrinking-Ck}\\
B_{\mathrm{new},k}^2
&:=
2(p-1)h_k^2M_1^2
+
\frac{h_k^2p\delta^2}{2c_k^2}.
\label{eq:noisy-shrinking-Bk}
\end{align}
Then, for every integer $K\ge0$,
\begin{align}
W_2^2(\nu_{K+1},\pi)
\le{}&
\left[
\prod_{j=0}^{K}(1-mh_j)
\right]
W_2^2(\nu_0,\pi)
\nonumber\\
&+
\sum_{i=0}^{K}
\left[
\prod_{j=i+1}^{K}(1-mh_j)
\right]
\left\{
\frac{C_i^2}{mh_i}
+
B_{\mathrm{new},i}^2
\right\}.
\label{eq:noisy-shrinking-W2-bound}
\end{align}
Consequently,
\begin{equation}
\lim_{K\to\infty}W_2(\nu_K,\pi)=0.
\label{eq:noisy-shrinking-convergence}
\end{equation}
\end{theorem}

\begin{proof}
At iteration \(k\), apply \eqref{eq:w2_onestep} of
Lemma~\ref{lem:one-step-w2} with \(h=h_k\) and \(c=c_k\).
This bound incorporates the \(p^2\) Taylor remainder from
Lemma~\ref{lem:spsa-error}. With \(C_k\) and \(B_{\mathrm{new},k}^2\)
as in \eqref{eq:noisy-shrinking-Ck} and \eqref{eq:noisy-shrinking-Bk},
we obtain
\begin{align}
W_2^2(\nu_{k+1},\pi)
\le{}&
\left[
(1-mh_k)W_2(\nu_k,\pi)+C_k
\right]^2
+
B_{\mathrm{new},k}^2.
\label{eq:shrinking-one-step-proof}
\end{align}
Since
\[
0<h_k\le\frac{2}{m+M},
\]
we have $0<mh_k\le1$. Young's inequality gives
\begin{align}
&2(1-mh_k)C_kW_2(\nu_k,\pi)
\nonumber\\
&\qquad\le
mh_k(1-mh_k)W_2^2(\nu_k,\pi)
+
\frac{1-mh_k}{mh_k}C_k^2.
\label{eq:shrinking-young-proof}
\end{align}
Because $0<mh_k\le 1$, applying $2ab\le a^2+b^2$ with
$a=\sqrt{mh_k}\,W_2(\nu_k,\pi)$ and $b=C_k/\sqrt{mh_k}$,
and multiplying by $1-mh_k\ge 0$, gives
\[
2(1-mh_k)C_k W_2(\nu_k,\pi)
\le (1-mh_k)mh_k W_2^2(\nu_k,\pi)
   + \frac{(1-mh_k)C_k^2}{mh_k}.
\]
Using this bound for the cross term and collecting terms, we obtain
\begin{align}
&\left[
(1-mh_k)W_2(\nu_k,\pi)+C_k
\right]^2
\nonumber\\
&\quad=
(1-mh_k)^2W_2^2(\nu_k,\pi)
+
2(1-mh_k)C_kW_2(\nu_k,\pi)
+
C_k^2
\nonumber\\
&\quad\le
(1-mh_k)W_2^2(\nu_k,\pi)
+
\frac{C_k^2}{mh_k}.
\label{eq:shrinking-square-bound}
\end{align}
Substituting \eqref{eq:shrinking-square-bound} into
\eqref{eq:shrinking-one-step-proof} yields
\begin{equation}
W_2^2(\nu_{k+1},\pi)
\le
(1-mh_k)W_2^2(\nu_k,\pi)
+
\frac{C_k^2}{mh_k}
+
B_{\mathrm{new},k}^2.
\label{eq:shrinking-linear-proof}
\end{equation}
Iterating \eqref{eq:shrinking-linear-proof} from \(0\) to \(K\)
gives the product and sum in \eqref{eq:noisy-shrinking-W2-bound}.
We next show that the bound vanishes. Since
\[
\sum_{k=0}^{\infty}h_k=\infty,
\]
Because $0 \le 1-mh_j \le e^{-mh_j}$ for every $j$ and
$\sum_{j=0}^{\infty} h_j = \infty$, the coefficient of the
initial error tends to zero:
\begin{equation}
  0 \le \prod_{j=0}^{K} (1-mh_j)
  \le \exp\!\left(-m\sum_{j=0}^{K} h_j\right)
  \longrightarrow 0
  \qquad \text{as } K \to \infty.
\label{eq:shrinking-contraction-proof}
\end{equation}
It remains to control the accumulated remainder terms. Applying $(a+b)^2\le 2a^2+2b^2$ to the two terms in
\[
C_k=\frac{7\sqrt{2}}{6}M(h_k^3p)^{1/2}
    +\frac{1}{3}h_kc_k^2M_3p^2
\]
and dividing by $mh_k$, we obtain
\begin{align}
\frac{C_k^2}{mh_k}
\le{}&
\frac{49M^2p}{9m}h_k^2
+
\frac{2M_3^2p^4}{9m}h_kc_k^4.
\label{eq:Ck-square-proof}
\end{align}
Therefore,
\begin{align}
&\frac{1}{h_k}
\left\{
\frac{C_k^2}{mh_k}
+
B_{\mathrm{new},k}^2
\right\}
\nonumber\\
&\quad\le
\frac{49M^2p}{9m}h_k
+
\frac{2M_3^2p^4}{9m}c_k^4
+
2(p-1)M_1^2h_k
+
\frac{p\delta^2}{2}\frac{h_k}{c_k^2}.
\label{eq:shrinking-remainder-ratio}
\end{align}
By assumption,
\[
h_k\longrightarrow0,
\qquad
c_k\longrightarrow0,
\qquad
{h_k}/{c_k^2}\longrightarrow0.
\]
Hence, the right-hand side of
\eqref{eq:shrinking-remainder-ratio} converges to zero. Thus,
\begin{equation}
\frac{C_k^2}{mh_k}
+
B_{\mathrm{new},k}^2
=
o(h_k).
\label{eq:shrinking-remainder-little-o}
\end{equation}
To conclude, fix an arbitrary $\eta>0$. By
\eqref{eq:shrinking-remainder-little-o}, there exists an integer
$k_\eta$ such that, for every $k\ge k_\eta$,
\begin{equation}
\frac{C_k^2}{mh_k}
+
B_{\mathrm{new},k}^2
\le
m\eta h_k.
\label{eq:shrinking-eta-remainder}
\end{equation}
It follows from \eqref{eq:shrinking-linear-proof} that
\begin{align}
W_2^2(\nu_{k+1},\pi)
\le{}&
(1-mh_k)W_2^2(\nu_k,\pi)
+
m\eta h_k.
\label{eq:shrinking-eta-recursion}
\end{align}
Subtracting $\eta$ from both sides of
\eqref{eq:shrinking-eta-recursion}, and using
\[
m\eta h_k-\eta
=
-(1-mh_k)\eta,
\]
we obtain the centered recursion
\begin{equation}
W_2^2(\nu_{k+1},\pi)-\eta
\le
(1-mh_k)
\left[
W_2^2(\nu_k,\pi)-\eta
\right].
\label{eq:shrinking-centered-recursion}
\end{equation}
Iterating this inequality from \(k_\eta\) to \(K\) yields
\begin{align}
W_2^2(\nu_{K+1},\pi)-\eta
\le{}&
\left[
\prod_{j=k_\eta}^{K}(1-mh_j)
\right]
\left[
W_2^2(\nu_{k_\eta},\pi)-\eta
\right].
\label{eq:shrinking-centered-iteration}
\end{align}
Equivalently,
\begin{align}
W_2^2(\nu_{K+1},\pi)
\le{}&
\left[
\prod_{j=k_\eta}^{K}(1-mh_j)
\right]
W_2^2(\nu_{k_\eta},\pi)
\nonumber\\
&+
\eta
\left[
1-
\prod_{j=k_\eta}^{K}(1-mh_j)
\right].
\label{eq:shrinking-final-eta-bound}
\end{align}
Since $\sum_k h_k=\infty$, the product in
\eqref{eq:shrinking-final-eta-bound} converges to zero. Therefore,
\[
\limsup_{K\to\infty}
W_2^2(\nu_{K+1},\pi)
\le\eta.
\]
Because $\eta>0$ is arbitrary,
\[
\lim_{K\to\infty}W_2(\nu_K,\pi)=0.
\]
\end{proof}

Theorem~\ref{thm:shrinking-noisy} establishes convergence for a broad class of stepsize and
perturbation sequences. To guide the choice of these sequences, we next derive
a schedule that balances the Taylor-remainder and measurement-noise terms
and prove that, for fixed $\delta>0$, it yields a $W_2$ convergence rate
of $O(k^{-1/3})$.

\begin{theorem}[A \(k^{-1/3}\) Wasserstein rate for \(\delta>0\)]
\label{thm:shrinking-rate-delta-positive}
Suppose that Assumptions A1--A5 hold with fixed \(\delta>0\) and
\(M_3>0\). Choose
\[
a>\frac{2}{3m},
\qquad
k_0\ge
\max\left\{
1,\frac{a(m+M)}{2}
\right\},
\]
and set
\begin{equation}
h_k:=\frac{a}{k+k_0},
\qquad
c_k^2:=
\left(
\frac{9m\delta^2h_k}{8M_3^2p^3}
\right)^{1/3}.
\label{eq:recommended-positive-delta-parameters}
\end{equation}
Then there exists a finite constant $D>0$, independent of $k$,
such that, for every $k\ge0$,
\begin{equation}
W_2^2(\nu_k,\pi)
\le
\frac{D}{(k+k_0)^{2/3}},
\qquad
W_2(\nu_k,\pi)
\le
\frac{\sqrt D}{(k+k_0)^{1/3}}
=
\mathcal O(k^{-1/3}).
\label{eq:positive-delta-wasserstein-rate}
\end{equation}
Up to problem-dependent constants, the leading noise-dependent
contribution to this bound scales as
$\mathcal O(\delta^{2/3}p\,k^{-1/3})$.
\end{theorem}
\begin{proof}
The choices in \eqref{eq:recommended-positive-delta-parameters}
satisfy \eqref{eq:shrinking-h-conditions} and
\eqref{eq:shrinking-c-conditions}. Combining the recursion
\eqref{eq:shrinking-linear-proof} from the proof of
Theorem~\ref{thm:shrinking-noisy} with its remainder estimate
\eqref{eq:shrinking-remainder-ratio}, we obtain
\begin{align}
W_2^2(\nu_{k+1},\pi)
&\le
(1-mh_k)W_2^2(\nu_k,\pi)
+
\left(
\frac{49M^2p}{9m}
+
2(p-1)M_1^2
\right)h_k^2
\nonumber\\
&\quad
+
\frac{2M_3^2p^4}{9m}h_kc_k^4
+
\frac{p\delta^2}{2}\frac{h_k^2}{c_k^2}.
\label{eq:theorem6-starting-recursion}
\end{align}
For fixed $h_k$, it remains to choose $c_k$ to balance the two $c_k$-dependent terms
\[
\frac{2M_3^2p^4}{9m}h_kc_k^4
+
\frac{p\delta^2}{2}\frac{h_k^2}{c_k^2}.
\]
Regarding the preceding expression as a function
of $c_k^2$, its unique minimizer is
\begin{equation}
c_k^2
=
\left(
\frac{9m\delta^2h_k}
     {8M_3^2p^3}
\right)^{1/3}.
\label{eq:optimal-ck-positive-delta-proof}
\end{equation}
At this choice,
\begin{align}
\frac{2M_3^2p^4}{9m}h_kc_k^4
+
\frac{p\delta^2}{2}\frac{h_k^2}{c_k^2}
&=
\frac{2M_3^2p^4}{3m}
\left(
\frac{9m\delta^2}
     {8M_3^2p^3}
\right)^{2/3}
h_k^{5/3}.
\label{eq:optimized-remainder-ratio}
\end{align}
Moreover, since $(h_k)_{k\ge0}$ is uniformly bounded and
$h_k^2=o(h_k^{5/3})$, there exists a constant $Q>0$, independent
of $k$, such that
\begin{equation}
\left(
\frac{49M^2p}{9m}
+
2(p-1)M_1^2
\right)h_k^2
+
\frac{2M_3^2p^4}{9m}h_kc_k^4
+
\frac{p\delta^2}{2}\frac{h_k^2}{c_k^2}
\le
Qh_k^{5/3}.
\label{eq:optimized-total-local-remainder}
\end{equation}
Combining \eqref{eq:theorem6-starting-recursion} and
\eqref{eq:optimized-total-local-remainder}, and substituting
$h_k=a/(k+k_0)$, gives
\begin{equation}
W_2^2(\nu_{k+1},\pi)
\le
\left(
1-\frac{ma}{k+k_0}
\right)
W_2^2(\nu_k,\pi)
+
\frac{Qa^{5/3}}{(k+k_0)^{5/3}}.
\label{eq:harmonic-wasserstein-recursion}
\end{equation}
Set
\begin{equation}
  D := \max\left\{
    k_0^{2/3} W_2^2(\nu_0,\pi),\,
    \frac{Q a^{5/3}}{ma-\tfrac{2}{3}}
  \right\}.
\label{eq:choice-of-induction-constant}
\end{equation}
This choice is well defined because \(a>2/(3m)\).
We prove by induction that
\begin{equation}
W_2^2(\nu_k,\pi)
\le
\frac{D}{(k+k_0)^{2/3}},
\qquad k\ge 0.
\label{eq:induction-claim-positive-delta}
\end{equation}
The claim holds at \(k=0\) by
\eqref{eq:choice-of-induction-constant}. Suppose that it holds
at iteration \(k\). Then, by
\eqref{eq:harmonic-wasserstein-recursion},
\begin{align}
W_2^2(\nu_{k+1},\pi)
&\le
\left(
1-\frac{ma}{k+k_0}
\right)
\frac{D}{(k+k_0)^{2/3}}
+
\frac{Qa^{5/3}}{(k+k_0)^{5/3}}
\nonumber\\
&=
\frac{D}{(k+k_0)^{2/3}}
-
\frac{maD-Qa^{5/3}}{(k+k_0)^{5/3}}.
\label{eq:induction-recursion-positive-delta}
\end{align}
By the second condition in
\eqref{eq:choice-of-induction-constant},
\[
\left(ma-\frac{2}{3}\right)D
\ge
Qa^{5/3},
\]
and hence
\[
maD-Qa^{5/3}
\ge
\frac{2D}{3}.
\]
Therefore,
\begin{equation}
W_2^2(\nu_{k+1},\pi)
\le
\frac{D}{(k+k_0)^{2/3}}
-
\frac{2D}{3(k+k_0)^{5/3}}.
\label{eq:induction-tangent-positive-delta}
\end{equation}
By convexity, the graph of $y=x^{-2/3}$ on $(0,\infty)$
lies above its tangent line at $x=k+k_0$. The slope there is
$-\tfrac{2}{3}(k+k_0)^{-5/3}$. Evaluating the tangent line
at $x=k+k_0+1$ therefore gives
\begin{align}
(k+k_0+1)^{-2/3}
&\ge
(k+k_0)^{-2/3}
-\frac{2}{3}(k+k_0)^{-5/3}.
\label{eq:convexity-power-minus-two-thirds}
\end{align}
Multiplying \eqref{eq:convexity-power-minus-two-thirds} by
\(D\) and combining it with
\eqref{eq:induction-tangent-positive-delta}, we obtain
\[
W_2^2(\nu_{k+1},\pi)
\le
\frac{D}{(k+k_0+1)^{2/3}}.
\]
Thus, the induction closes, and
\eqref{eq:induction-claim-positive-delta} holds for every
\(k\ge 0\).
Taking square roots gives
\[
W_2(\nu_k,\pi)
\le
\frac{\sqrt D}{(k+k_0)^{1/3}}.
\]
\end{proof}

For a constant stepsize, we bound the finite-sample Wasserstein error and the number of function evaluations needed to attain a target accuracy. We also give explicit choices of \(h_k\) and \(c_k\) that yield an
\(O(k^{-1/3})\) Wasserstein convergence rate. We now examine how sampling error varies with dimension and compare
LMC-SPSA under shrinking schedules with LMC-FDSA and the ZO-LMC method of Roy et al.~\cite{Roy2022} under a common function-evaluation budget.

\section{Numerical Experiments}
\subsection{Dimension-Dependence validation: \(O(p^2)\) Scaling}

We first validate the sharpened dimension dependence predicted by our theory. 
The target potential is a quadratic
\begin{equation}
    f(x) \;=\; \tfrac{1}{2}\sum_{i=1}^{p} (x_i - 2)^2,
    \qquad \pi(x)\propto e^{-f(x)} ,
\end{equation}
i.e., a Gaussian with mean \(2\cdot \mathbf{1}_p\).
We work in the \emph{noisy-oracle} regime where only noisy log-likelihood values 
\(
\widehat f(x)=f(x)+\zeta, \zeta\sim\mathcal N\!\big(0,\sigma^2 p\big),
\) are available. LMC-SPSA forms a local score approximation 
from two such noisy evaluations per iteration. For each dimension \(p\), we run the sampler
and estimate the terminal error \(W_2(\nu_K,\pi)\); we then 
plot the normalized curve \(W_2(\nu_K,\pi)/p^2\) versus \(p\).

\noindent
\textbf{Interpretation and relevance.}
As shown in Fig.~\ref{fig:w2p2_vs_p}, the normalized error \(W_2/p^2\) exhibits a clear
flattening trend with increasing \(p\), consistent with our \(O(p^2)\) bound. We adopt this model for applicability: in many real scenarios (e.g., black-box simulators,
privacy/sensor noise), \(\nabla f\) is unavailable while only
noisy function evaluations can be measured.

\begin{figure}[H]
    \centering
    
    \includegraphics[width=0.5\linewidth]{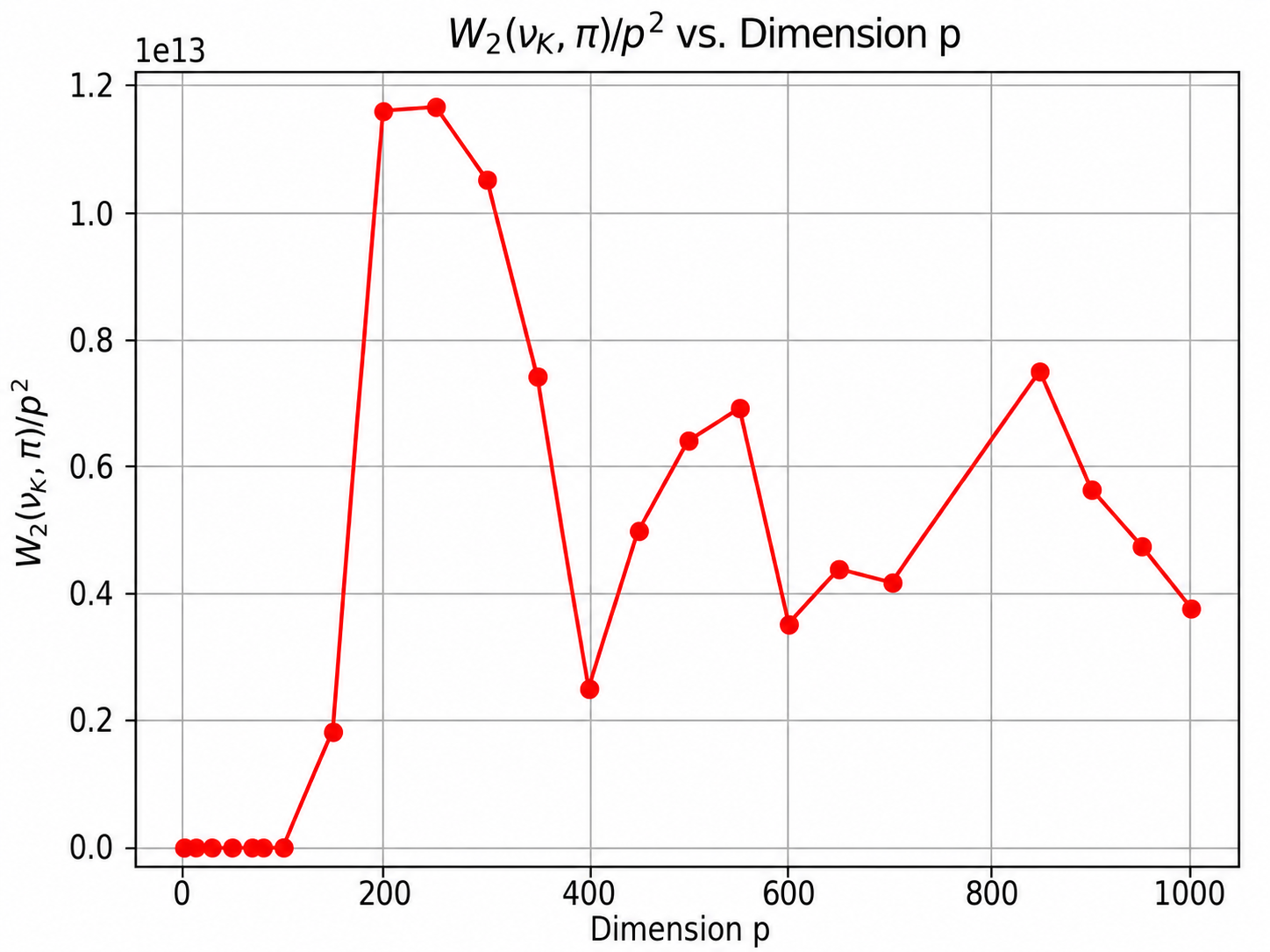}
    \caption{\(W_2(\nu_K,\pi)/p^2\) versus dimension \(p\).
    Fluctuations at moderate \(p\) are due to finite-iteration and noise effects.}
    \label{fig:w2p2_vs_p}
\end{figure}
\subsection{Vanishing schedules and oracle efficiency under a fixed call budget}
\label{sec:vanishing_oracle_efficiency}
A key implication of our refined non-asymptotic \(W_2\) bounds (Theorem~1) is that
\emph{constant} stepsize/perturbation generally induces a non-vanishing error floor
(from Euler discretization and gradient-estimation bias),
whereas a properly chosen \emph{shrinking} schedule \((h_k,c_k)\) yields a
\emph{vanishing guarantee} in the sense that the additive remainder in the \(W_2\) upper bound
goes to zero as \(k\to\infty\).
This section provides a call-budget comparison that makes both the
\emph{vanishing guarantee} and the \emph{oracle efficiency} (calls per iteration) explicit.

\paragraph{Target and oracle model.}
We consider a strongly log-concave target \(\pi\) on \(\mathbb{R}^p\) with potential \(f\),
and access \(f\) through a noisy function-value oracle with measurement noise satisfying Assumption A5.
All methods start from the same deterministic initialization \(x_0\), and are compared under a
\emph{fixed oracle-call budget} \(B\) (total number of noisy function evaluations).

\paragraph{Algorithms and oracle cost per iteration.}
We compare:
(i) LMC-SPSA (Algorithm~1),
(ii) LMC-FDSA (finite-difference variant),
and (iii) the one-point ZO-LMC method of \cite{Roy2022}.
A critical distinction is the number of oracle calls per iteration:
\[
\text{LMC-SPSA: }2\ \text{calls/iter},\qquad
\text{LMC-FDSA: }2p\ \text{calls/iter},\qquad
\text{Roy: }2b\ \text{calls/iter},
\]
where \(b\) is the batch size used in \cite{Roy2022}.
Therefore, under the same call budget \(B\), the effective iteration counts scale as
\[
K_{\mathrm{SPSA}}\approx \Big\lfloor \tfrac{B}{2}\Big\rfloor,\qquad
K_{\mathrm{FDSA}}\approx \Big\lfloor \tfrac{B}{2p}\Big\rfloor,\qquad
K_{\mathrm{Roy}}(b)\approx \Big\lfloor \tfrac{B}{2b}\Big\rfloor.
\]
This \emph{oracle efficiency} effect is unavoidable: larger \(p\) (FDSA) or larger \(b\) (Roy)
reduces the number of Langevin updates that can be executed within a fixed call budget.

\paragraph{Theoretically justified parameter choices.}
For LMC-SPSA and LMC-FDSA, we use shrinking schedules in the admissible theoretical regime.
In contrast, for Roy's method we follow their paper and use the \emph{constant} tuning recommended
by \cite{Roy2022} (Theorem~4.2): constant stepsize \(h\), smoothing radius \(\nu\),
and batch size \(b\). We report \(b\in\{1,4,128\}\) as well as the theoretical choice \(b^\star\)
(from \cite{Roy2022}) to illustrate the variance--iteration tradeoff.

\paragraph{Diagnostics and plotting protocol.}
We compare all methods under a fixed oracle-call budget \(B\) and summarize
performance over \(R\) independent chains. Let
\(\{X_s\}_{s\geq 1}\subset\mathbb{R}^p\) denote the post-burn-in samples
for a given method, and let \(t(c)\) be the number of post-burn-in samples
available by oracle-call count \(c\). We index all diagnostics by oracle calls
to ensure fair comparisons across methods with different numbers of function
evaluations per iteration. For each chain, define the post-burn-in cumulative mean and its squared error by
\[
\widehat{\mu}_{\mathrm{cum}}(c)
\triangleq \frac{1}{t(c)}\sum_{s=1}^{t(c)} X_s,
\qquad
e_{\mathrm{cum}}(c)
\triangleq \bigl\|\widehat{\mu}_{\mathrm{cum}}(c)-2\cdot\mathbf{1}_p\bigr\|_2^2.
\]
We estimate the MSE by averaging these squared errors across the \(R\) chains.
For the moment diagnostics, write
\(\widehat{\mu}_j(c)=\frac{1}{t(c)}
\sum_{s=1}^{t(c)}X_{s,j}\) and
\[
\begin{aligned}
\widehat{\operatorname{Var}}_j(c)
&=
\frac{1}{t(c)}
\sum_{s=1}^{t(c)}X_{s,j}^2
-\widehat{\mu}_j(c)^2,
\\[-0.15em]
\widehat{m}_{4,j}(c)
&=
\widehat{\mathbb{E}}_c[X_j^4]
-4\widehat{\mu}_j(c)\widehat{\mathbb{E}}_c[X_j^3]
+6\widehat{\mu}_j(c)^2\widehat{\mathbb{E}}_c[X_j^2]
-3\widehat{\mu}_j(c)^4.
\end{aligned}
\]
where
\[
\widehat{\mathbb{E}}_c[X_j^r]
=
\frac{1}{t(c)}
\sum_{s=1}^{t(c)}X_{s,j}^r,
\qquad r=2,3,4.
\]
For the target \(\mathcal{N}(2\mathbf{1}_p,I_p)\), we report the
coordinate-averaged signed deviations
\[
\begin{aligned}
\operatorname{Dev}_1(c)
&=
\frac{1}{p}\sum_{j=1}^p
\bigl(\widehat{\mu}_j(c)-2\bigr),
&
\operatorname{Dev}_2(c)
&=
\frac{1}{p}\sum_{j=1}^p
\bigl(\widehat{\operatorname{Var}}_j(c)-1\bigr),
\\[-0.15em]
\operatorname{Dev}_4(c)
&=
\frac{1}{p}\sum_{j=1}^p
\bigl(\widehat{m}_{4,j}(c)-3\bigr).
\end{aligned}
\]
For the MSE, we plot the mean across the \(R\) chains together with the
interquartile range (25th--75th percentiles). For the signed moment diagnostics,
we plot the ensemble mean with a horizontal zero line indicating exact
agreement with the corresponding target moment. The signs of
\(\operatorname{Dev}_2\) and \(\operatorname{Dev}_4\) indicate under- or
over-dispersion and lighter- or heavier-than-Gaussian tails, respectively.
\begin{figure}[t]
  \centering

  \begin{subfigure}[t]{0.8\linewidth}
    \centering
    \includegraphics[width=\linewidth]{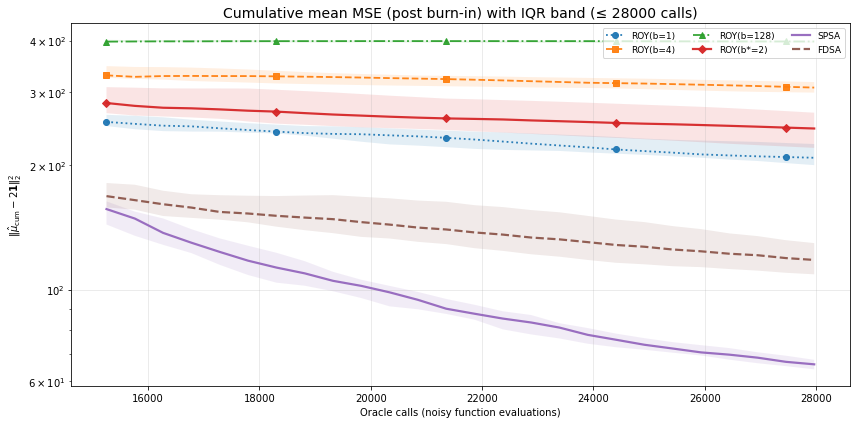}
    \caption{\textbf{MSE of the post-burn-in cumulative mean.}
Curves show the mean squared error across chains, and shaded bands
span the 25th to 75th percentiles of the chainwise squared errors.}
    \label{fig:shrink_oracle_mse}
  \end{subfigure}

  \vspace{0.6em}

  \begin{subfigure}[t]{0.98\linewidth}
    \centering
    \includegraphics[width=\linewidth]{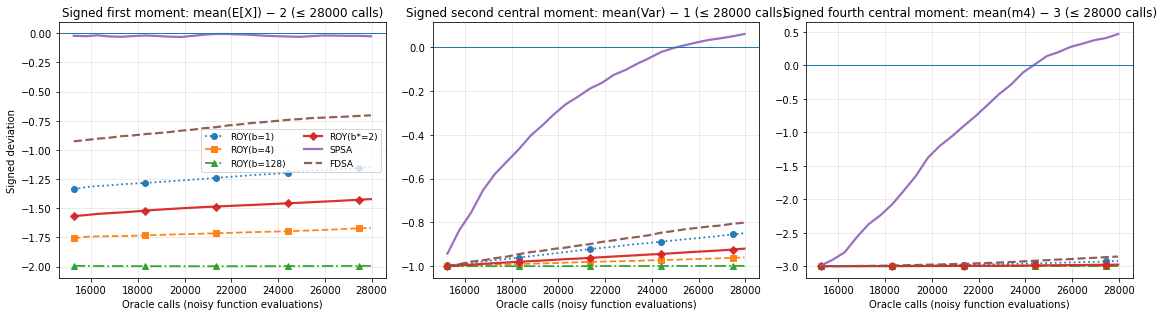}
    \caption{\textbf{Signed cumulative moment deviations} (post burn-in, coordinate-averaged). 
    }
    \label{fig:shrink_oracle_moments}
  \end{subfigure}

  \caption{\textbf{Oracle-budget comparison under shrinking schedules (truncated to $\leq 28{,}000$ oracle calls).}
  }
  \label{fig:shrink_oracle_combo}
\end{figure}

\paragraph{Results: vanishing behavior and the \(b\)-tradeoff.}
Figure~\ref{fig:shrink_oracle_combo} highlights two effects.
First, the shrinking schedules
produce sustained improvement in the cumulative mean MSE across the oracle-call budget,
consistent with our \emph{vanishing guarantee} (the additive remainder term in the \(W_2\) upper bound decreases to zero).
Second, oracle efficiency strongly shapes practical performance under a fixed budget:
Roy's method performs best for small \(b\) (more Langevin updates under the same \(B\)),
whereas large \(b\) (e.g., \(b=128\)) can stagnate because it executes too few iterations.
The theoretically recommended \(b^\star\) provides a principled middle ground among Roy variants.
Under equal oracle calls, LMC-SPSA remains the most competitive overall due to its two-call gradient estimator
combined with a shrinking-parameter schedule.
The signed moment diagnostics complement the MSE view by indicating the direction of moment deviations
(under/over-dispersion and light/heavy tails) while trending toward zero as the oracle budget increases.

\subsection{Dimension scaling under a fixed oracle-call budget}
\label{subsec:dim_scaling_fixed_budget}

We complement the main call-by-call trajectories with a dimension-scaling study under a \emph{fixed} oracle-call budget.
We consider the quadratic potential
$f(x)=\tfrac12\|x-2\cdot\mathbf{1}\|_2^2$, so the target is $\pi=\mathcal{N}(2\cdot\mathbf{1},I_p)$, and we use a noisy function-value oracle
\[
\widehat f(x)=f(x)+\zeta,\qquad \zeta\sim\mathcal{N}(0,\sigma^2 p),
\]
with $\sigma=0.1$. For each dimension $p\in\{5,10,20,30,50,70,100\}$, we allocate a total budget of $B=10^4$ noisy function evaluations
and run each method from $X_0=0$ for $20$ independent seeds. Performance is summarized by the median over seeds of
(i) the absolute errors in the first raw moment and the second
and fourth central moments of the first coordinate, and
(ii) the squared Euclidean error
\(\|\widehat{\mu}-\mu\|_2^2\) of the empirical mean.
For each run, we discard the first half of the trajectory as
burn-in and compute all metrics using the remaining samples.

\paragraph{Findings.}
Figure~\ref{fig:dim_scaling_metrics} reports the scaling of all four metrics with $p$ at fixed $B$.

\noindent Across this range of dimensions, LMC-SPSA consistently attains the smallest (or comparable smallest) errors and exhibits the best scaling in $p$,
which is consistent with the fact that it preserves the shrinking (vanishing) schedule while spending only $2$ oracle calls per step.
In contrast, LMC-FDSA deteriorates faster with dimension under the same budget due to its $2p$ calls/iteration requirement. Roy's method exhibits larger errors over the moderate-to-large dimensional
regime under the same oracle budget, even though $b$ remains small
(here $b=1$ for $p\le 50$ and $b=2$ for $p\ge 70$). 
A log-log regression of the mean-squared error suggests
\[
\mathrm{MSE}\approx p^{1.336}\ \text{(SPSA)},\qquad
\mathrm{MSE}\approx p^{1.642}\ \text{(FDSA)},\qquad
\mathrm{MSE}\approx p^{1.731}\ \text{(Roy)},
\]
highlighting the advantage of the shrinking SPSA scheme under increasing dimension.

\begin{figure}[H]
    \centering
    \includegraphics[width=0.98\linewidth]{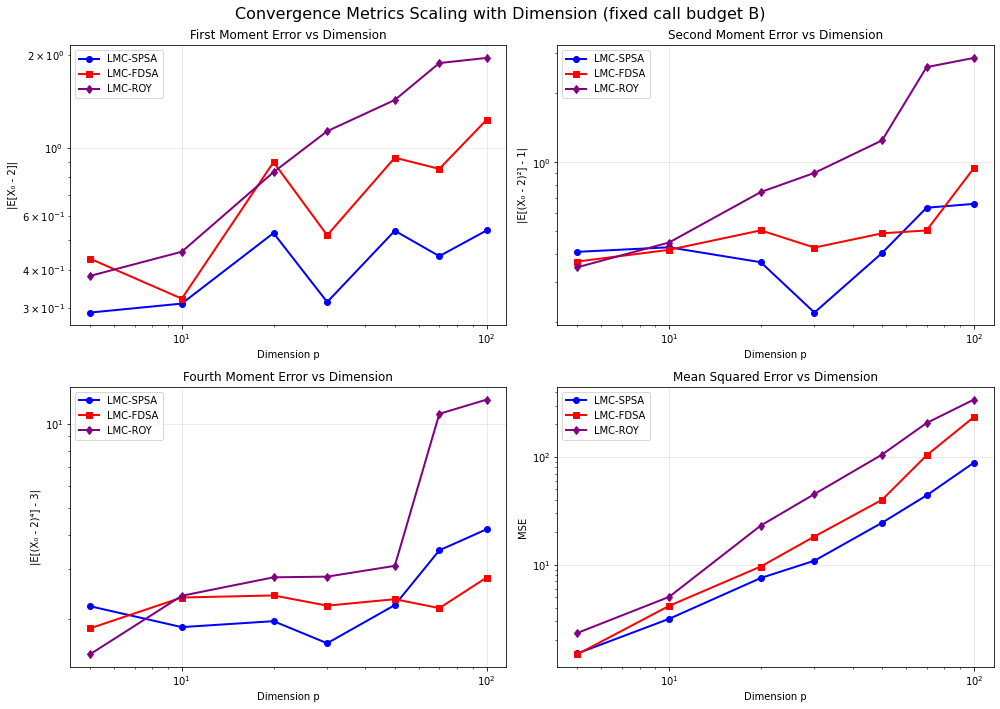}
    \vspace{-0.5em}
    \caption{Scaling with dimension \(p\) under a fixed oracle-call
budget \(B=10^4\). For each method and dimension, the curves report
the median over 20 independent runs after discarding the first half
of each trajectory as burn-in. The first three panels show the
absolute errors in the first/second/fourth moments of the first coordinate. The final panel shows the
squared Euclidean error
\(\|\widehat{\mu}-\mu\|_2^2\) of the empirical mean.}
    \label{fig:dim_scaling_metrics}
\end{figure}

\section{Conclusion}
We analyzed LMC-SPSA for sampling with two noisy function
evaluations per iteration. For a fixed stepsize and perturbation
size, we derived a non-asymptotic bound on the Wasserstein error.
Under Assumption~A3, we reduced the dimension dependence of the
higher-order SPSA approximation term from quartic in the analysis
of Sun and Spall~\cite{sun23} to quadratic.

We used this bound to choose the fixed parameters and determine
how many noisy function evaluations suffice for a prescribed
Wasserstein accuracy. Under additive measurement
noise with fixed positive variance, the noise-dependent term
in our oracle complexity bound grows cubically with inverse
accuracy, compared with quartic growth for the ZO-LMC method
of Roy et al.~\cite{Roy2022}. This improvement in accuracy
dependence comes with cubic rather than quadratic dependence
on dimension. If both evaluations share the same additive
measurement noise, the oracle-complexity
guarantee improves further because the paired noise cancels within
each SPSA difference.

We extended the fixed-stepsize analysis to LMC-SPSA with both
a diminishing stepsize and a diminishing perturbation size.
We proved that the Wasserstein error tends to zero even under
persistent measurement noise. We also derived an explicit rule
for choosing the two parameters together: as the stepsize
decreases, the squared perturbation size should be proportional
to the cube root of the stepsize. This choice balances
finite-difference bias against noise amplification and, with our
stepsize schedule, yields a Wasserstein error bound that decays
as the inverse cube root of the iteration count.

Finally, under matched noisy-function-evaluation budgets, the
reported experiments show that LMC-SPSA achieves lower
mean-estimation errors and more stable higher-moment diagnostics
than LMC-FDSA and Roy et al.'s ZO-LMC across the settings
considered. These theoretical and empirical results support
LMC-SPSA as a query-efficient sampling method for noisy black-box
targets.

\textbf{Future work.}
The variance-reduction ideas explored in Hamiltonian Monte Carlo---multiple perturbations per step (mSPSA) and control variates for SPSA based on stochastic variance-reduced gradient (SVRG) methods~\cite{hiruntiaranakul2024spsahmc}---are natural candidates for extension to \emph{LMC-SPSA with noise}.
Two open questions are particularly relevant under our \emph{matched-budgets} protocol:
(i) when do the variance reductions justify the additional
function evaluations per iteration and yield a lower MSE
at a fixed total evaluation budget;
(ii) how to adapt our nonasymptotic $W_2$ analysis to account for the reduced gradient variance while preserving the $O(p^2)$ dimension term.
We leave a systematic development of mSPSA-/SVRG-enhanced LMC-SPSA to future work.


\begin{thebibliography}{99}
\setlength{\itemsep}{0.35em}
\setlength{\parsep}{0pt}

\bibitem{brooks2011handbook}
S.~Brooks, A.~Gelman, G.~L.~Jones, and X.-L.~Meng, Eds.,
\emph{Handbook of Markov Chain Monte Carlo}.
Chapman \& Hall/CRC, 2011.

\bibitem{robert2004montecarlo}
C.~P.~Robert and G.~Casella,
\emph{Monte Carlo Statistical Methods}, 2nd ed.
Springer, 2004.

\bibitem{metropolis1953}
N.~Metropolis, A.~W.~Rosenbluth, M.~N.~Rosenbluth, A.~H.~Teller, and E.~Teller,
``Equation of state calculations by fast computing machines,''
\emph{Journal of Chemical Physics},
vol.~21, no.~6, pp.~1087--1092, 1953.

\bibitem{hastings1970}
W.~K.~Hastings,
``Monte Carlo sampling methods using Markov chains and their applications,''
\emph{Biometrika},
vol.~57, no.~1, pp.~97--109, 1970.

\bibitem{jiang2023sbi}
H.~Jiang, Y.~Wang, and Y.~Yang,
``Simulation-based inference via Langevin dynamics with score matching,''
arXiv preprint arXiv:2509.03853, 2025.

\bibitem{andrieu2010pmcmc}
C.~Andrieu, A.~Doucet, and R.~Holenstein,
``Particle Markov chain Monte Carlo methods,''
\emph{Journal of the Royal Statistical Society: Series B (Statistical Methodology)},
vol.~72, no.~3, pp.~269--342, 2010.

\bibitem{carpenter2017stan}
B.~Carpenter, A.~Gelman, M.~D.~Hoffman, D.~Lee, B.~Goodrich,
M.~Betancourt, M.~Brubaker, J.~Guo, P.~Li, and A.~Riddell,
``Stan: A probabilistic programming language,''
\emph{Journal of Statistical Software},
vol.~76, no.~1, pp.~1--32, 2017.

\bibitem{ronquist2003mrbayes}
F.~Ronquist and J.~P.~Huelsenbeck,
``MrBayes 3: Bayesian phylogenetic inference under mixed models,''
\emph{Bioinformatics},
vol.~19, no.~12, pp.~1572--1574, 2003.

\bibitem{neal2011hmc}
R.~M.~Neal,
``MCMC using Hamiltonian dynamics,''
in \emph{Handbook of Markov Chain Monte Carlo},
S.~Brooks, A.~Gelman, G.~L.~Jones, and X.-L.~Meng, Eds.
Chapman \& Hall/CRC, 2011, ch.~5, pp.~113--162.

\bibitem{welling2011sgld}
M.~Welling and Y.~W.~Teh,
``Bayesian learning via stochastic gradient Langevin dynamics,''
in \emph{Proceedings of the 28th International Conference on Machine Learning (ICML)},
2011, pp.~681--688.

\bibitem{roberts1996langevin}
G.~O.~Roberts and R.~L.~Tweedie,
``Exponential convergence of Langevin distributions and their discrete approximations,''
\emph{Bernoulli},
vol.~2, no.~4, pp.~341--363, 1996.

\bibitem{roberts1998optimal}
G.~O.~Roberts and J.~S.~Rosenthal,
``Optimal scaling of discrete approximations to Langevin diffusions,''
\emph{Journal of the Royal Statistical Society: Series B (Statistical Methodology)},
vol.~60, no.~1, pp.~255--268, 1998.

\bibitem{cheng2018underdamped}
X.~Cheng, N.~S.~Chatterji, P.~L.~Bartlett, and M.~I.~Jordan,
``Underdamped Langevin MCMC: A non-asymptotic analysis,''
in \emph{Proceedings of the 31st Conference on Learning Theory (COLT)},
2018, pp.~300--323.

\bibitem{sun23}
S.~Sun and J.~C.~Spall,
``Langevin Monte Carlo with SPSA-approximated gradients,''
in \emph{Proceedings of the 57th Annual Conference on Information Sciences and Systems (CISS)},
2023, pp.~1--6.

\bibitem{chartrand2011}
R.~Chartrand,
``Numerical differentiation of noisy, nonsmooth data,''
\emph{ISRN Applied Mathematics},
vol.~2011, Art.~ID 164564, 11 pp., 2011.

\bibitem{spall92}
J.~C.~Spall,
``Multivariate stochastic approximation using a simultaneous perturbation gradient approximation,''
\emph{IEEE Transactions on Automatic Control},
vol.~37, no.~3, pp.~332--341, 1992.

\bibitem{blakney2019fdvsspsa}
A.~Blakney and J.~Zhu,
``A comparison of the finite difference and simultaneous perturbation gradient estimation methods with noisy function evaluations,''
in \emph{Proceedings of the 53rd Annual Conference on Information Sciences and Systems (CISS)},
2019, pp.~1--6.

\bibitem{spall94}
J.~C.~Spall,
``Developments in stochastic optimization algorithms with gradient approximations based on function measurements,''
in \emph{Proceedings of the 1994 Winter Simulation Conference},
1994, pp.~207--214.

\bibitem{spall05}
J.~C.~Spall,
\emph{Introduction to Stochastic Search and Optimization: Estimation, Simulation, and Control}.
John Wiley \& Sons, 2003.

\bibitem{chen2021arxiv}
Y.~Chen,
``Theoretical study and comparison of SPSA and RDSA algorithms,''
arXiv preprint arXiv:2107.12771, 2021.

\bibitem{peng2023acc}
D.~Peng, Y.~Chen, and J.~C.~Spall,
``Formal comparison of simultaneous perturbation stochastic approximation and random direction stochastic approximation,''
in \emph{Proceedings of the 2023 American Control Conference (ACC)},
2023, pp.~744--749.

\bibitem{dingli2021rcdvr}
Z.~Ding and Q.~Li,
``Langevin Monte Carlo: random coordinate descent and variance reduction,''
\emph{Journal of Machine Learning Research},
vol.~22, no.~205, pp.~1--51, 2021.

\bibitem{spall1997onemeasurement}
J.~C.~Spall,
``A one-measurement form of simultaneous perturbation stochastic approximation,''
\emph{Automatica},
vol.~33, no.~1, pp.~109--112, 1997.

\bibitem{liu2020onepointzosgld}
L. Liu and Z. Wang, ``One-point gradient estimators for
zeroth-order stochastic gradient Langevin dynamics,'' in
\emph{OPT 2020: 12th Annual Workshop on Optimization for Machine
Learning}, 2020. [Online]. Available:
\url{https://opt-ml.org/oldopt/papers/2020/paper_96.pdf}

\bibitem{pavliotis2022}
G.~A.~Pavliotis, A.~M.~Stuart, and U.~Vaes,
``Derivative-free Bayesian inversion using multiscale dynamics,''
\emph{SIAM Journal on Applied Dynamical Systems},
vol.~21, no.~1, pp.~284--326, 2022.

\bibitem{bao2025}
E.~Bao, Y.~Jiang, F.~Wei, X.~Xiao, Z.~Li, Y.~Li, and B.~Ding,
``Unlocking the power of differentially private zeroth-order optimization for fine-tuning LLMs,''
in \emph{Proceedings of the 34th USENIX Security Symposium (USENIX Security 25)},
2025, pp.~1569--1588.

\bibitem{song2019score}
Y.~Song and S.~Ermon,
``Generative modeling by estimating gradients of the data distribution,''
in \emph{Advances in Neural Information Processing Systems},
vol.~32, 2019, pp.~11895--11907.

\bibitem{dalalyan19}
A.~S.~Dalalyan and A.~Karagulyan,
``User-friendly guarantees for the Langevin Monte Carlo with inaccurate gradient,''
\emph{Stochastic Processes and their Applications},
vol.~129, no.~12, pp.~5278--5311, 2019.

\bibitem{hiruntiaranakul2024spsahmc}
S.~Hiruntiaranakul,
\emph{SPSA-augmented Hamiltonian Monte Carlo},
Master's thesis, Johns Hopkins University, Baltimore, MD, USA, 2024.
[Online]. Available:
\url{https://jscholarship.library.jhu.edu/bitstreams/fd72f4af-722d-4aed-afcb-94ea9d69616e/download}

\bibitem{Roy2022}
A.~Roy, L.~Shen, K.~Balasubramanian, and S.~Ghadimi,
``Stochastic zeroth-order discretizations of Langevin diffusions for Bayesian inference,''
\emph{Bernoulli},
vol.~28, no.~3, pp.~1810--1834, 2022.

\bibitem{cao2011acc}
X.~Cao,
``Preliminary results on non-Bernoulli distribution of perturbations for simultaneous perturbation stochastic approximation,''
in \emph{Proceedings of the 2011 American Control Conference (ACC)},
2011, pp.~2669--2670.

\bibitem{maeda1997}
Y.~Maeda and R.~J.~P.~de~Figueiredo,
``Learning rules for neuro-controller via simultaneous perturbation,''
\emph{IEEE Transactions on Neural Networks},
vol.~8, no.~5, pp.~1119--1130, 1997.

\bibitem{durmus2017}
A.~Durmus and E.~Moulines,
``Nonasymptotic convergence analysis for the unadjusted Langevin algorithm,''
\emph{The Annals of Applied Probability},
vol.~27, no.~3, pp.~1551--1587, 2017.

\end{thebibliography}
\end{document}